\documentclass[10pt]{article}

\usepackage[T1]{fontenc}
\usepackage{lmodern}
\usepackage{amsmath,amssymb,amsthm,mathtools,mathrsfs}
\usepackage{geometry}
\usepackage{microtype}
\usepackage{enumitem}
\usepackage{xcolor}
\usepackage{hyperref}
\usepackage{placeins}
\usepackage{graphicx}

\hypersetup{
colorlinks=true,
hypertexnames=false,
linkcolor=black,
citecolor=black,
urlcolor=black,
pdftitle={Autonomous Flows: Exact Finite Interpolation and Uniform Approximation Obstructions},
pdfauthor={Qiaomu Liu, Qi L\"u, Enrique Zuazua},
pdfsubject={Exact finite interpolation and uniform approximation obstructions for autonomous flows},
pdfkeywords={neural ordinary differential equations, universal approximation, exact finite-data interpolation, simultaneous cell controllability, autonomous flows, shallow ReLU networks, uniform neighborhood-routing obstruction}
}
\setlist[enumerate]{leftmargin=2.1em,itemsep=1.5pt,topsep=4pt,parsep=0pt}
\setlist[itemize]{leftmargin=2em,itemsep=1.5pt,topsep=4pt,parsep=0pt}

\AtBeginDocument{%
\setlength{\abovedisplayskip}{3.5pt plus 0.8pt minus 0.8pt}%
\setlength{\belowdisplayskip}{3.5pt plus 0.8pt minus 0.8pt}%
\setlength{\abovedisplayshortskip}{2pt plus 0.5pt minus 0.5pt}%
\setlength{\belowdisplayshortskip}{2.5pt plus 0.5pt minus 0.5pt}%
\setlength{\jot}{1.5pt}%
}
\allowdisplaybreaks
\mathtoolsset{showonlyrefs}

\newtheorem{theorem}{Theorem}[section]
\newtheorem{proposition}[theorem]{Proposition}
\newtheorem{lemma}[theorem]{Lemma}
\newtheorem{corollary}[theorem]{Corollary}
\numberwithin{equation}{section}
\theoremstyle{definition}
\newtheorem{definition}[theorem]{Definition}
\theoremstyle{remark}
\newtheorem{remark}[theorem]{Remark}

\newcommand{\R}{\mathbb R}
\newcommand{\Nzero}{\mathbb N_0}
\newcommand{\NN}{\mathscr N}
\newcommand{\End}{\mathcal E}
\newcommand{\cD}{\mathcal D}
\newcommand{\cV}{\mathcal V}
\newcommand{\cA}{\mathcal A}
\newcommand{\B}{\mathrm B}
\newcommand{\Lip}{\operatorname{Lip}}
\newcommand{\Str}{\operatorname{Str}}
\newcommand{\sep}{\operatorname{sep}}
\newcommand{\diam}{\operatorname{diam}}
\newcommand{\supp}{\operatorname{supp}}
\newcommand{\Diff}{\operatorname{Diff}}
\newcommand{\Id}{\operatorname{Id}}
\newcommand{\Var}{\operatorname{Var}}
\newcommand{\op}{\mathrm{op}}
\newcommand{\TV}{\mathrm{TV}}
\newcommand{\dd}{\,\mathrm d}
\newcommand{\abs}[1]{\left|#1\right|}
\newcommand{\norm}[1]{\left\lVert#1\right\rVert}
\newcommand{\aug}[1]{#1}

\title{\textbf{Autonomous Flows: Exact Finite Interpolation and Uniform Approximation Obstructions}\\[2mm]
}
\author{Qiaomu Liu\thanks{School
	of Mathematics, Sichuan University, Chengdu
	610064, China. (Email: {\tt liuqm@stu.scu.edu.cn}). Qiaomu Liu is supported by the NSF of China under grant 12025105.} \and Qi L\"u\thanks{School
	of Mathematics, Sichuan University, Chengdu
	610064, China. (Email: {\tt lu@scu.edu.cn}). Qi L\"u is supported by the NSF of China under grant 12025105. }  \and Enrique Zuazua\thanks{\aug{Corresponding author. Chair for Dynamics, Control, Machine Learning and Numerics (Alexander von Humboldt Professorship), Department of Mathematics, Friedrich--Alexander--Universit\"at Erlangen--N\"urnberg, Cauerstrasse 11, 91058 Erlangen, Germany; Chair of Computational Mathematics, Fundaci\'on Deusto, Avenida de las Universidades 24, 48007 Bilbao, Spain; Departamento de Matem\'aticas, Universidad Aut\'onoma de Madrid, Calle Francisco Tom\'as y Valiente 7, 28049 Madrid, Spain. (Email: {\tt enrique.zuazua@fau.de}).}}}
\date{}

\begin{document}
\maketitle

\begin{abstract}
In dimension at least two, a class of locally Lipschitz vector fields realizes every finite correspondence between distinct inputs and distinct targets at any prescribed positive time, provided that it is linear, its members generate global flows, and it approximates every smooth vector field of bounded support uniformly on bounded sets. The input and target sets may overlap. The proof constructs a smooth autonomous reference flow and finitely many localized correction fields, then uses uniform stability and Brouwer degree to obtain exact endpoints within the approximating class. The argument uses only uniform approximation of the vector fields; it does not require control of their derivatives. Applied to shallow ReLU vector fields, the result gives exact finite interpolation without time-dependent coefficients or additional state variables, together with bounds on width and normalized coefficient strength. In contrast, no continuous autonomous semiflow can exchange and compress two disjoint balls. Its time maps therefore fail to approximate all continuous maps uniformly on compact sets. The results distinguish exact interpolation at finitely many points from uniform control of neighborhoods.
\end{abstract}



\section{Introduction}

Finite-data interpolation asks whether a prescribed class of maps can carry each
point of one finite configuration to a prescribed point of another.  Neural ODEs
provide one motivating setting \cite{chen2018,weinan2017}.  A
control-theoretic treatment of classification, approximation, and transport for
neural ODEs \cite{ruizbalet2023} provides an important starting point for the
line of work to which the present paper contributes.  For an autonomous flow, the terminal map has
the form $F=\Phi_V^T$: one time-independent vector field must move all samples
during one common time interval.  This is substantially more rigid than
arbitrary finite-point matching by diffeomorphisms, since a
diffeomorphism carrying one configuration to another need not be the time-$T$
map of an autonomous flow \cite{arangoGomez2002}.

The main result of this paper is that this rigidity does not prevent exact finite
interpolation in dimension $N\ge2$.  If a linear class of locally Lipschitz
autonomous vector fields generates global flows and can approximate smooth compactly supported fields arbitrarily well on bounded sets,
then every correspondence between
finite configurations with pairwise distinct source points and pairwise distinct
target points is realized exactly at every prescribed positive time; see
Theorem~\ref{thm:abstract-exactification}.  This includes points
that remain fixed, assignments in which the target of one point is another input
point, and finite cycles such as $x_1\mapsto x_2\mapsto x_3\mapsto x_1$.  Thus
the source and target configurations need not be disjoint.  Pairwise distinct
targets are necessary because a flow map is injective.

\paragraph{Proof strategy.}
The difficulty is to make all endpoint errors vanish while keeping the vector field in the prescribed class. We proceed in four steps.
\begin{enumerate}
\item \textbf{Construct a smooth reference flow.} Complete the finite correspondence into cycles and fixed points. Realize these by planar rotations and stationary points, then transport the model to the prescribed configuration by a smooth change of coordinates. After time rescaling, this reference flow reaches all targets at time one.
\item \textbf{Construct independent endpoint corrections.} Each moving sample has a short trajectory segment that no other marked sample visits during the unit time interval. Perturbations near these segments, together with bumps near stationary samples, vary every endpoint coordinate independently to first order at the smooth reference field.
\item \textbf{Approximate the correction fields first.} Choose correction fields from the prescribed class close enough to the ideal ones that their endpoint response matrix, evaluated at the smooth reference field, remains invertible. With these fields fixed, choose a small coefficient ball and then approximate the reference field accurately enough to control the endpoint error throughout that ball.
\item \textbf{Remove the endpoint error by degree.} On the boundary of the coefficient ball, the endpoint error differs from the invertible linear response by less than the size of that response. Brouwer degree therefore gives coefficients for which every endpoint error vanishes. Linearity keeps the corrected field in the prescribed class.
\end{enumerate}
The key point is that uniform approximation suffices: we do not approximate derivatives or require differentiability of the endpoint map at a ReLU field. The variational calculation is performed at the smooth reference field, and the final step uses continuity and degree. For the quantitative bounds, we estimate perturbed trajectories in the standard rotation coordinates. This avoids an exponential loss involving the potentially large Lipschitz constant of the transported field. Section~\ref{sec:exactification} gives the construction and estimates.

\paragraph{A ReLU application.}
Finite sums of ReLU ridge fields are globally Lipschitz and can
approximate smooth vector fields arbitrarily well on bounded sets.  Hence one time-independent shallow ReLU field
exactly interpolates every admissible finite dataset in $\mathbb R^N$ for
$N\ge2$, at any prescribed positive time, without temporal switching or state
augmentation.  Theorem~\ref{thm:quantitative-exactification} gives quantitative
bounds on width and time--strength; Proposition~\ref{prop:risk-floor} gives a
complementary bounded-distortion lower bound.  In dimension one, order
preservation rules out universal exact interpolation.

\paragraph{Rigidity at neighborhood scale.}
The flexibility at the point scale has a sharp limitation at the neighborhood
scale.  For a radius fixed before the field is chosen, a universal routing
requirement would ask that each source ball be sent into a smaller ball around
its assigned target.  Theorem~\ref{thm:two-cell} shows that no continuous
autonomous semiflow can exchange and compress two disjoint balls in this way:
the semigroup identity and Brouwer's fixed-point theorem force a contradiction.
The obstruction is structural and does not exclude data-dependent neighborhoods
obtained after a particular interpolating flow has been selected.
The distinction between pointwise and neighborhood control is a
central motivation of recent work \cite{alvarez2026}: simultaneous control of neighborhoods or cells, rather than
pointwise interpolation alone, is the stronger property relevant to universal
approximation and generalization estimates.  Corollary~\ref{cor:uniform-uap-obstruction}
consequently rules out uniform density of autonomous time maps in the continuous
maps on compact sets. Thus exact interpolation of finitely many
points does not supply the neighborhood control required for those stronger
consequences; indeed, autonomy obstructs every universal fixed-radius version
of it.  This concerns the uniform topology and the full class
$C(\mathcal K;\mathbb R^N)$, not weaker topologies or almost-everywhere approximation. Corollary~\ref{cor:uniform-uap-obstruction} does not characterize which invertible targets are approximable.

\paragraph{Relation to previous work.}
Finite-ensemble controllability and approximation on mapping spaces have been
studied through Lie-algebraic, time-dependent, and switched mechanisms
\cite{agrachevLetrouit2026,agrachev2020,agrachev2022,grong2026,gugat2026}.
These mechanisms do not face the same autonomous semigroup constraint.
Earlier autonomous neural-ODE interpolation results impose additional geometric
or architectural conditions \cite{alvarez2024}.  Recent work
\cite{alvarez2026} formulates the remaining gap that motivates the present study: it develops pointwise
interpolation and simultaneous neighborhood control in its semi-autonomous
setting, while the corresponding general exact-interpolation question for a
single-layer autonomous system remains open there.  By contrast, the criterion
proved here applies to every linear, globally well-posed,
class with this approximation property; shallow ReLU fields are its principal
application.  Approximation of diffeomorphisms is related but distinct:
geometric finite-point relocation supplies diffeomorphisms moving finite
configurations in dimension at least two \cite{banyaga1997,michorVizman1994},
whereas autonomous embeddability remains a separate issue
\cite{arangoGomez2002}.  Further comparisons are given in
Section~\ref{sec:learning-consequences}.

The shallow autonomous ReLU model used quantitatively is
\[
\dot z(t)=V(z(t)),\qquad
V(x)=\sum_{j=1}^p w_j(a_j^\top x+b_j)_+,
\qquad (s)_+:=\max\{s,0\}.
\]
Positive homogeneity and coefficient-controlled approximation enter only in the
quantitative estimates.  Section~\ref{sec:setting-main} states the main results;
Section~\ref{sec:exactification} proves the endpoint-correction principle;
Section~\ref{sec:quantitative} derives the ReLU bounds; and
Section~\ref{sec:cell-limitations} proves the obstructions.  The appendices
contain the geometric, approximation, bridge, and closure estimates.

Figure~\ref{fig:dimension-routing} summarizes the dimension- and scale-dependent mechanisms behind these results.

\begin{figure}[!ht]
\centering
\includegraphics[width=\textwidth]{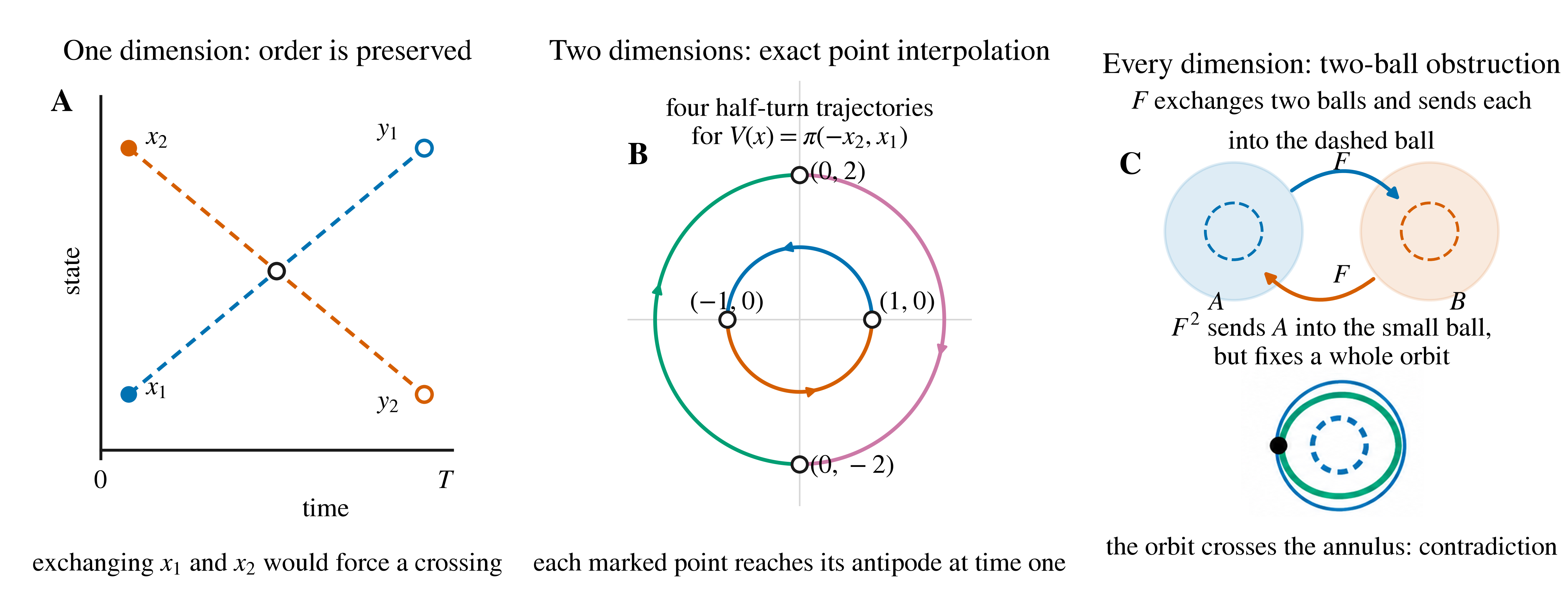}
\caption{Exact interpolation and the obstruction to routing balls. (A) A one-dimensional flow preserves order, so it cannot exchange two ordered inputs. (B) The field $V(x)=\pi(-x_2,x_1)$ rotates the inputs $(\pm1,0)$ and $(0,\pm2)$ to their antipodes at time one; the four marked trajectories lie on two circles. (C) If $F=\Psi^T$ exchanged and compressed two disjoint closed balls, Brouwer's theorem would give a fixed point of $F^2$ in the first ball. Its semiflow trajectory would cross the surrounding annulus, where $F^2$ cannot both fix the crossing point and send it into the smaller target ball. Panel B shows the explicit example; panels A and C illustrate the arguments.}
\label{fig:dimension-routing}
\end{figure}

\section{Framework and principal results}
\label{sec:setting-main}

Several architecture classes organize the comparison. An \emph{autonomous} model uses one field $V(x)$, so its terminal map repeats under iteration. A \emph{semi-autonomous} model has restricted explicit time dependence, for example affine time dependence in thresholds; equivalently, some such schedules can be encoded autonomously after adjoining a clock state. A \emph{fully nonautonomous} model uses $V(t,x)$ and may select distinct temporal stages, although this freedom alone does not imply simultaneous cell controllability. An \emph{augmented-state autonomous} model evolves in a larger state space and can evade restrictions tied to the original-state topology; augmented neural ODEs were introduced to alleviate original-state representational restrictions \cite{dupont2019}. Controlled residual dynamics and continuous normalizing flows supply discrete-depth and invertible-flow comparisons, respectively \cite{behrmann2019,grathwohl2019}. The distinctions relevant here are exact finite interpolation, fixed-radius neighborhood routing, invertibility, and the availability of temporal scheduling.

In the autonomous shallow network, the trainable ``controls'' are fixed coefficients of one time-independent vector field; once chosen, they act throughout the entire interval. Nonautonomous control architectures can instead select different operations at different temporal stages. This distinction explains both sides of the paper: endpoint sensitivities permit simultaneous correction of finitely many terminal constraints, while autonomy forces the semigroup identity $\Phi^{2T}=\Phi^T\circ\Phi^T$ on neighborhood behavior.

\subsection{Framework: model and expressivity notions}

\subsubsection{Autonomous ReLU flows and finite data}

Let \(\Nzero:=\{0,1,2,\ldots\}\), and set \(\log_+(r):=\max\{\log r,0\}\) for \(r>0\), with \(\log_+(0):=0\). 
For \(d\ge1\), \(x\in\R^d\), and \(r>0\), write
\[
\B_r^d(x):=\{z\in\R^d:|z-x|<r\},
\qquad
\B_r^d:=\B_r^d(0).
\]
When \(d=N\), we suppress the superscript and write
\(\B_r(x)\) and \(\B_r\). Thus \(\B_r\) is reserved for a ball, whereas the italic symbol \(B\) below denotes the linear part of an affine normalization.  For \(N\ge1\) and \(p\in\Nzero\), define
\[
\NN_{N,p}:=
\left\{
x\longmapsto\sum_{j=1}^p w_j(a_j^\top x+b_j)_+:
w_j,a_j\in\R^N,\ b_j\in\R
\right\},
\qquad
\NN_N:=\bigcup_{p\in\Nzero}\NN_{N,p}.
\]
For a map \(u\) on a set \(A\), let
\(\norm{u}_{C(A)}:=\sup_{x\in A}|u(x)|\), and let \(\Lip_A(u)\) denote its
least Lipschitz constant on \(A\). For an integer \(s\ge0\),
\(\norm{u}_{C^s(A)}\) is the maximum of the suprema of all derivatives of order
at most \(s\), when they exist, while \(\norm{D^q u(x)}_{\op}\) denotes the
operator norm of the \(q\)-linear derivative. Finally,
\(\Diff_c^\infty(\Omega)\) denotes the group of smooth diffeomorphisms of
\(\Omega\) that agree with \(\Id\) outside a compact subset of \(\Omega\).

The empty sum is the zero field.  Each \(V\in\NN_N\) is globally Lipschitz and has at most linear growth.  It therefore generates a unique global flow \(\Phi_V^t\), \(t\in\R\).

For a parameter list \(\theta=((w_j,a_j,b_j))_{j=1}^p\), write
\[
V_\theta(x):=\sum_{j=1}^p w_j(a_j^\top x+b_j)_+,
\qquad p(\theta):=p.
\]
Thus \(V_\theta\in\NN_N\); different parameter lists may represent the same field.

Let \(M\in\mathbb N\), set \(m:=NM\), and let \(\cD=\{(x_i,y_i)\}_{i=1}^M\subset\R^N\times\R^N\). We identify \((\R^N)^M\) with \(\R^m\) and equip it with the Euclidean norm. The dataset is \emph{admissible} if the inputs \(x_i\) are pairwise distinct and the targets \(y_i\) are pairwise distinct.  Fixed constraints \(x_i=y_i\) and cross-relations \(y_i=x_j\) are allowed.  For pairwise distinct inputs, pairwise distinct targets are necessary because every flow map is injective; repeated identical constraints may be discarded.
Set \(Y:=(y_1,\ldots,y_M)\in\R^m\).
\begin{remark}[Overlapping input and target configurations]
Admissibility does not require the union of inputs and targets to contain $2M$ points. Stationary constraints $x_i=y_i$, cross-relations $y_i=x_j$, finite cycles, and partial overlap of the input and target sets are all permitted. The synchronized reference construction forms the directed correspondence on the $K$ distinct vertices, completes its paths into disjoint cycles when necessary, and assigns stationary vertices to fixed components. Thus repetitions in the union create shared vertices, not conflicting endpoint conditions. The construction is carried out in Sections~\ref{sec:graph-completion}--\ref{sec:smooth-reference-module}, and the endpoint directions for moving and stationary samples are constructed in Section~\ref{sec:endpoint-frame-module}; the regularity and derivative estimates are deferred to Appendix~\ref{app:finite-transitivity}.
\end{remark}
\begin{remark}\label{rem:one-dimensional-obstruction}
The restriction \(N\ge2\) in the exact interpolation theorem is structural, rather than a technical artifact of the proof.  In one dimension, every flow map preserves order: if \(x<x'\), then \(\Phi_V^t(x)<\Phi_V^t(x')\) for every \(t\in\R\). Indeed, two distinct trajectories cannot meet by uniqueness of solutions. Hence the admissible assignment \(x_1=0,\ x_2=1,\ y_1=1,\ y_2=0\) cannot be realized by any one-dimensional flow.  Thus universal exact interpolation already fails for \(N=1\).  In dimensions \(N\ge2\), trajectories have transverse room to route around one another; this geometric freedom is what makes arbitrary finite assignments possible.
\end{remark}

\subsubsection{Point interpolation and uniform neighborhood routing}

Fix \(T>0\). For any locally Lipschitz vector field \(V\) generating a global flow, define the terminal, or endpoint, map---the list of all sample locations after time $T$---by \(\End(V):=(\Phi_V^T(x_1),\ldots,\Phi_V^T(x_M))\in\R^m\). For an admissible dataset, let
\[
p_{\min}(\cD,T):=
\inf\left\{
p(\theta):
\Phi_{V_\theta}^T(x_i)=y_i,\ i=1,\ldots,M
\right\},
\]
with the convention that the infimum is \(+\infty\) if the constraint set is empty.

A \emph{continuous autonomous semiflow}---a continuous forward-time evolution whose successive time intervals compose--- is a family \((\Psi^t)_{t\ge0}\) of self-maps such that \((t,x)\mapsto\Psi^t(x)\) is continuous on \([0,\infty)\times\R^N\), \(\Psi^0=\Id\), and \(\Psi^{t+s}=\Psi^t\circ\Psi^s\) for \(s,t\ge0\).

\begin{definition}[Universal point-to-neighborhood property]
Let \(\mathscr F\) be a class of autonomous vector fields with global flows.
\begin{enumerate}[label=(\roman*)]
\item The class \(\mathscr F\) has property \((\mathsf P_0)\) if every admissible finite dataset is interpolated exactly at time \(T\) by a field in \(\mathscr F\).
\item For \(\delta>0\), the class \(\mathscr F\) has property \((\mathsf P_\delta)\) if, for every finite family of centers \(x_1,\ldots,x_M\) for which the balls \(\overline\B_\delta(x_i)\) are pairwise disjoint, every family of pairwise distinct target centers \(y_1,\ldots,y_M\), and every \(0<\varepsilon<\delta\), there exists \(V\in\mathscr F\) such that
\[
\Phi_V^T(\overline\B_\delta(x_i))
\subset\B_\varepsilon(y_i),
\qquad i=1,\ldots,M.
\]
\end{enumerate}
\end{definition}

The input radius in \((\mathsf P_\delta)\) is fixed before the vector field is chosen.  This order of quantifiers distinguishes universal fixed-radius neighborhood routing from the data-dependent neighborhoods supplied by continuity after a particular interpolating field has been fixed.

We use the notion of \emph{simultaneous cell controllability} (SCC) introduced in~\cite{alvarez2026}: a control architecture has SCC if it can send every finite family of pairwise disjoint compact convex cells into arbitrarily small balls about prescribed targets. SCC includes the two-ball routing configurations used below. For an architecture closed under positive time rescaling, it also implies \((\mathsf P_\delta)\) at each prescribed terminal time.

\subsection{Principal structural results}
\label{sec:results}

The paper has two principal structural statements. Theorem~\ref{thm:abstract-exactification} gives the positive approximation-to-interpolation principle, and Theorem~\ref{thm:two-cell} gives the complementary autonomous semigroup obstruction. Their shallow-ReLU and approximation-theoretic consequences are stated alongside them; the quantitative ReLU refinement is separated below from the two structural statements.

\subsubsection{From approximation to exact interpolation}

\begin{theorem}[Approximation on bounded sets implies exact finite interpolation] 
\label{thm:abstract-exactification}
Let $N\ge2$, and let
$\mathscr F\subset C_{\mathrm{loc}}^{0,1}(\R^N;\R^N)$
be a linear space of vector fields such that every $V\in\mathscr F$ generates a global flow.  Assume that members of $\mathscr F$ can approximate every compactly supported smooth vector field arbitrarily well on every bounded set: for every $u\in C_c^\infty(\R^N;\R^N)$, $R>0$, and $\varepsilon>0$, there exists $V\in\mathscr F$ such that
\[
\norm{V-u}_{C(\overline\B_R)}<\varepsilon.
\]
Then, for every $T>0$ and every admissible finite correspondence $\{(x_i,y_i)\}_{i=1}^M$, there exists $V\in\mathscr F$ such that $\Phi_V^T(x_i)=y_i$ for all $i$.
\end{theorem}

The proof is given in Section~\ref{sec:exactification}; the regularity and quantitative estimates needed by the ReLU application are recorded in the appendices.

\paragraph{Shallow-ReLU consequence.}

\begin{corollary}[Exact finite interpolation by shallow autonomous ReLU flows]
\label{thm:qualitative-interpolation}
For every $N\ge2$, every admissible finite correspondence in $\mathbb R^N$, and every prescribed time $T>0$, there exists one time-independent field
\[
V(x)=\sum_{j=1}^p w_j(a_j^\top x+b_j)_+
\]
whose time-$T$ flow satisfies $\Phi_V^T(x_i)=y_i$ for all $i$.  This includes stationary constraints, cross-relations, finite cycles, and arbitrary overlap between the input and target configurations.
\end{corollary}

This follows from Theorem~\ref{thm:abstract-exactification}, as detailed in Section~\ref{sec:relu-qualitative-application}.

\begin{remark}[Reusable criterion beyond ReLU]
\label{rem:beyond-relu}
Theorem~\ref{thm:abstract-exactification} is not specific to ReLU: any linear dictionary class with the same approximation-on-bounded-sets property and global well-posedness assumptions inherits exact finite interpolation.  In particular, globally Lipschitz nonpolynomial ridge activations, such as the logistic sigmoid or $\tanh$, provide natural examples, since their finite ridge expansions are globally Lipschitz and classical universal approximation theorems provide the required approximation on bounded sets \cite{cybenko1989,hornik1991,pinkus1999}.  Thus the qualitative exact-interpolation result extends to these activations.  The quantitative width and normalized-strength estimates below, however, remain ReLU-specific because they use coefficient-controlled approximation and positive homogeneity.  
\end{remark}

\subsubsection{Autonomous semigroup obstruction}

\begin{theorem}[Semigroup obstruction to two-cell routing]
\label{thm:two-cell}
The following statement contains no neural-network hypothesis.
Let \(N\ge1\), \(T>0\), and \(\delta>0\).  Suppose \(\abs{a-b}>2\delta\) and \(0<\varepsilon<\delta\).  No continuous autonomous semiflow \((\Psi^t)_{t\ge0}\) on \(\R^N\) can satisfy
\[
\Psi^T(\overline\B_\delta(a))\subset\B_\varepsilon(b),
\qquad
\Psi^T(\overline\B_\delta(b))\subset\B_\varepsilon(a).
\]
\end{theorem}

\begin{corollary}[Sharp point-to-neighborhood threshold]
\label{cor:sharp-threshold}
Let \(N\ge2\), \(T>0\), and \(\delta\ge0\).  Then
\[
\NN_N\text{ has property }(\mathsf P_\delta)
\quad\Longleftrightarrow\quad
\delta=0.
\]
In particular, no original-state autonomous architecture whose admissible evolutions are continuous semiflows can satisfy SCC.
\end{corollary}

\begin{corollary}[Failure of uniform universal approximation by autonomous time maps]
\label{cor:uniform-uap-obstruction}
Let $N\geq1$ and $T>0$. There exist a compact ball
$\mathcal K\subset\mathbb R^N$, a map
$f\in C_c^\infty(\mathbb R^N;\mathbb R^N)$, and a constant $c>0$ such that
\[
\norm{\Psi^T-f}_{C(\mathcal K)}\geq c
\]
for every continuous autonomous semiflow $(\Psi^t)_{t\geq0}$ on $\mathbb R^N$.
Consequently, time-$T$ maps of continuous autonomous semiflows are not dense in
$C(\mathcal K;\mathbb R^N)$ in the uniform topology.
\end{corollary}
\begin{proof}
Fix $\delta>0$ and choose $a,b\in\mathbb R^N$ with $|a-b|>4\delta$.
There are disjoint open neighborhoods of $\overline\B_\delta(a)$ and
$\overline\B_\delta(b)$ and smooth compactly supported cutoff functions
$\eta_a,\eta_b$ that equal one on the respective closed balls and whose
supports lie in those neighborhoods. Set $f:=b\eta_a+a\eta_b$, and let
$\mathcal K$ be a closed ball containing both supports. Then
$f\equiv b$ on $\overline\B_\delta(a)$ and
$f\equiv a$ on $\overline\B_\delta(b)$.

If $\norm{\Psi^T-f}_{C(\mathcal K)}<\delta$, choose
$\varepsilon$ strictly between this norm and $\delta$. The uniform estimate on
the two closed source balls then gives
\[
\Psi^T(\overline\B_\delta(a))\subset\B_\varepsilon(b),
\qquad
\Psi^T(\overline\B_\delta(b))\subset\B_\varepsilon(a),
\]
contradicting Theorem~\ref{thm:two-cell}. Hence the asserted bound holds with
$c=\delta$.
\end{proof}

Combining Remark~\ref{rem:one-dimensional-obstruction}, Corollary~\ref{thm:qualitative-interpolation}, and Theorem~\ref{thm:two-cell} gives the sharp point-to-neighborhood picture. Universal exact point interpolation fails in dimension one by order preservation, while for the shallow autonomous ReLU class it holds in every dimension $N\ge2$; at every fixed positive radius, however, universal neighborhood routing fails for continuous autonomous semiflows. Equivalently, for $N\ge2$,
\[
\NN_N\text{ has property }(\mathsf P_\delta)
\quad\Longleftrightarrow\quad
\delta=0.
\]
Thus the positive and negative results meet at the point scale: exact finite interpolation is compatible with autonomy, whereas universal fixed-radius routing is not.

\subsection{Quantitative refinement for shallow ReLU flows}
\label{sec:quantitative-refinement}

For the quantitative statement, set
\[
\cV:=\{x_1,\ldots,x_M,y_1,\ldots,y_M\}_{\rm distinct},
\qquad K:=\abs{\cV}.
\]
For \(K\ge2\), define
\[
\sep(\cV):=\min_{v\ne v'}\abs{v-v'},
\qquad
\diam(\cV):=\max_{v,v'}\abs{v-v'},
\]
and the affine aspect ratio
\begin{equation}
\mathfrak a(\cV):=
\inf_{L_{\rm geom}\in\mathrm{GL}(N)}
\frac{\diam(L_{\rm geom}\cV)}{\sep(L_{\rm geom}\cV)}.
\end{equation}
For \(K=1\), set \(\mathfrak a(\cV)=1\).

For an affine normalization \(S(x)=B(x-x_*)\), \(B\in\mathrm{GL}(N)\), and a parameter list
\(\theta=((w_j,a_j,b_j))_{j=1}^p\), define
\begin{equation}
\begin{aligned}
\Str_S(\theta)
:=
\sum_{j=1}^p\abs{Bw_j}
\sqrt{
\abs{B^{-\top}a_j}^2+
\abs{a_j^\top x_*+b_j}^2
},
\qquad
\norm{\theta}_{S,2}^2
:=
\sum_{j=1}^p\left(
\abs{Bw_j}^2+
\abs{B^{-\top}a_j}^2+
\abs{a_j^\top x_*+b_j}^2
\right).
\end{aligned}
\end{equation}
The main bound is stated in terms of these two representation-dependent quantities. Their functional interpretation, the positive-homogeneity balancing identity, and the approximation exponents used in the proof are introduced later, at the point where they enter the argument.

\begin{theorem}[Quantitative shallow-ReLU realization]
\label{thm:quantitative-exactification}
For every \(N\ge2\), there exists \(C_N\ge2\), depending only on \(N\), with the following property.  Let \(T>0\), and let \(\cD=\{(x_i,y_i)\}_{i=1}^M\) be admissible.  There exist \(B\in\mathrm{GL}(N)\), \(x_*\in\R^N\), the affine map \(S(x)=B(x-x_*)\), and a parameter list \(\theta\) such that
\begin{align}
&\Phi_{V_\theta}^T(x_i)=y_i,\quad i=1,\ldots,M, \label{eq:main-exactness}\\
&\max\{p(\theta),T\Str_S(\theta)\}\le \mathfrak U_N(\cD):= \bigl(C_NK\mathfrak a(\cV)\bigr)^{C_NK}.
\end{align}
If \(K\ge2\), the affine map can be chosen so that
\begin{equation}
S(\cV)\subset\overline\B_{1/2},
\qquad
\sep(S(\cV))\ge\frac{1}{4\mathfrak a(\cV)}.
\label{eq:main-normalization}
\end{equation}
If \(K=1\), one may take \(S=\Id\) and the empty parameter list.  In all cases, the units may be balanced without changing the represented field, so that
\begin{equation}
\label{eq:main-balanced}
\norm{\theta}_{S,2}
\le
\sqrt{\frac{2\mathfrak U_N(\cD)}{T}}.
\end{equation}
In particular,
\begin{equation}
\label{eq:main-affine}
p_{\min}(\cD,T)\le\mathfrak U_N(\cD).
\end{equation}
\end{theorem}

The proof combines the finite-dimensional endpoint-correction construction of Section~\ref{sec:exactification}, the quantitative bridge of Proposition~\ref{prop:quantitative-interface}, and the coefficient-controlled ReLU approximation and width--strength accounting of Section~\ref{sec:quantitative}; the bridge estimates are established in Appendix~\ref{app:stability-degree}.

\FloatBarrier

\section{Finite-dimensional endpoint correction for autonomous flows}
\label{sec:exactification}

We follow the four steps described in the introduction. Section~\ref{sec:reference-construction} constructs a smooth autonomous reference field, and Section~\ref{sec:endpoint-frame-module} constructs correction fields with identity first-order endpoint response. Section~\ref{sec:stability-degree} shows that approximate correction fields retain an invertible response at the smooth reference field, controls the error after approximation of the reference field, and removes the residual by Brouwer degree. Section~\ref{sec:proof-main-theorem} combines these ingredients within the prescribed class. Quantitative bounds on the geometric construction are deferred to Appendix~\ref{app:finite-transitivity} and used in the ReLU refinement.

For the fixed-order regularity estimates used later, set
\begin{equation}
\label{eq:fixed-regularity-order}
s_N:=\left\lfloor\frac{N+3}{2}\right\rfloor+2.
\end{equation}
The order $s_N$ is used to index the derivative bounds for the transported reference field and endpoint directions.  It is not an assumption of Theorem~\ref{thm:abstract-exactification}: the qualitative exactification argument itself uses only smoothness, while the explicit $C^{s_N}$ bounds are recorded in Appendix~\ref{app:finite-transitivity} for the quantitative analysis of Section~\ref{sec:quantitative}.

\subsection{Smooth autonomous reference realizations}
\label{sec:reference-construction}

\subsubsection{Directed correspondence graphs and cycle completion}
\label{sec:graph-completion}

Associate with the dataset the directed graph whose distinct vertices form
\[
\cV:=\{x_1,\ldots,x_M,y_1,\ldots,y_M\}_{\rm distinct},
\]
and whose prescribed edges are \(x_i\to y_i\).

\begin{lemma}[Cycle completion of an admissible correspondence]
Every connected component of the directed data graph is a directed path, a directed cycle, or a self-loop.  By adding one unconstrained edge from the terminal vertex to the initial vertex of each directed path, the graph becomes a disjoint union of directed cycles and self-loops.  No prescribed edge is changed and no additional constrained sample is introduced.
\end{lemma}

\begin{proof}
Pairwise distinct inputs imply that every vertex has out-degree at most one, while pairwise distinct targets imply that every vertex has in-degree at most one.  A finite directed graph with these two properties has only path, cycle, and self-loop components.  A nonclosed path has a unique initial vertex with no incoming edge and a unique terminal vertex with no outgoing edge.  Adding the single terminal-to-initial edge closes that component without changing the outgoing edge of any vertex that already carries a prescribed assignment.  Repeating this independently for every path gives the claimed disjoint cycle decomposition.
\end{proof}

Figure~\ref{fig:correspondence-completion} illustrates the three component types and the auxiliary edge used to close a path.

\begin{figure}[!ht]
\centering
\begingroup\setlength{\fboxsep}{0pt}\setlength{\fboxrule}{0.8pt} \includegraphics[width=.8\textwidth]{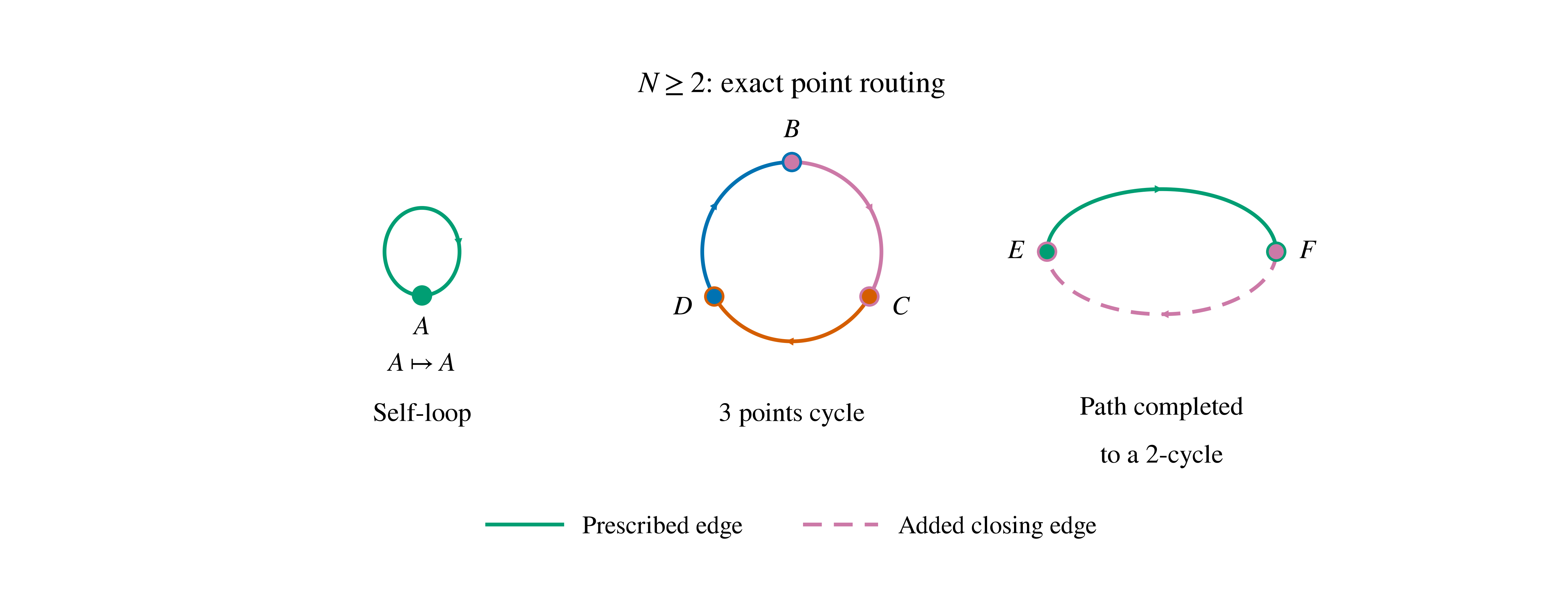}\endgroup
\caption{The three component types used in the reference construction.  Left: a stationary constraint is represented by a self-loop.  Center: a prescribed three-point cycle.  Right: a directed path is completed to a two-cycle by one auxiliary closing edge (dashed).  The auxiliary edge changes no original endpoint constraint.  The displayed curves are schematic reference routes at the data vertices, not the standard-model trajectories of $z_{\nu,k}$ under $f_0$ introduced below.}
\label{fig:correspondence-completion}
\end{figure}

\subsubsection{Standard cyclic and stationary models}

Index the completed components by positive integers \(\nu\), put \(c_\nu:=4\nu e_1\), and let a nontrivial component contain \(n_\nu\ge2\) vertices.  Place its standard representatives at
\[
z_{\nu,k}=c_\nu+\cos(2\pi k/n_\nu)e_1
+\sin(2\pi k/n_\nu)e_2,
\qquad k=0,\ldots,n_\nu-1.
\]
Let \(\mathsf R_{\rm plane}e_1=e_2\), \(\mathsf R_{\rm plane}e_2=-e_1\), and \(\mathsf R_{\rm plane}e_j=0\) for \(j\ge3\).  Fix \(\chi\in C_c^\infty((1/2,3/2))\), with \(0\le\chi\le1\) and \(\chi\equiv1\) on \([3/4,5/4]\), and define
\begin{equation}
f_{0,\nu}(z)
:=\omega_\nu\chi(|z-c_\nu|)\mathsf R_{\rm plane}(z-c_\nu),
\qquad
\omega_\nu:=\frac{2\pi}{n_\nu}.
\end{equation}
Place each self-loop at the corresponding center \(c_\nu\), and let \(f_0\) be the sum of the component fields.  Their supports are disjoint.  On every standard cycle, \(\chi=1\), so the flow is the planar rotation of angular speed \(2\pi/n_\nu\); hence
\[
\Phi_{f_0}^1(z_{\nu,k})=z_{\nu,k+1\,({\rm mod}\,n_\nu)}.
\]
A standard self-loop lies at a center where \(f_0=0\), and is therefore fixed. Thus the completed correspondence is realized exactly by one smooth autonomous standard field. For each fixed derivative order \(s\), the norm \(\norm{f_0}_{C^s}\) is bounded by a constant depending only on \(N\), \(s\), and the fixed cutoff template, independently of the number of components. Indeed, the component supports are disjoint, their angular speeds satisfy \(\omega_\nu\le\pi\), and translation of the fixed cutoff does not change its derivative bounds.

\subsubsection{Finite-point relocation by localized point pushes}

The standard vertices must now be transported to the actual data configuration.  The following quantitative finite-transitivity statement is also the geometric input for the later upper bound.

\begin{lemma}[Quantitative finite-point relocation] 
\label{lem:finite-transitivity}
Fix \(N\ge2\), \(s\ge2\), \(K\ge1\), and \(R\ge1\). Let \(\mathbf z=(z_1,\ldots,z_K)\) and \(\mathbf w=(w_1,\ldots,w_K)\) be ordered configurations of distinct points in \(\B_R\), and put
\[
d_{\rm cfg}:=
\begin{cases}
1,&K=1,\\
\min\{1,\sep(\{z_j\}),\sep(\{w_j\})\},&K\ge2.
\end{cases}
\]
If an open set \(\Omega_{\rm tr}\) contains \(\overline\B_{R+4K+4}\), there is \(H\in\Diff_c^\infty(\Omega_{\rm tr})\) such that \(H(z_j)=w_j\) and
\begin{equation}
\label{eq:H-bound}
1+\max_{1\le q_{\rm der}\le s}
\left(
\norm{D^{q_{\rm der}}H}_\infty+
\norm{D^{q_{\rm der}}H^{-1}}_\infty
\right)
\le
\left(
\frac{C_{\mathrm{tr},N,s}(1+K)(1+R)}{d_{\rm cfg}}
\right)^{C_{\mathrm{tr},N,s}K},
\end{equation}
where \(C_{\mathrm{tr},N,s}\ge2\) depends only on \(N\) and \(s\).
\end{lemma}

\begin{proof}[Proof sketch]
Choose mutually separated parking points outside the ball containing the source and target configurations. Move the source points to the parking points one at a time, and then move the parking points to the targets one at a time; all points not currently moving are treated as obstacles.  For each individual move from \(a\) to \(b\), one can choose a distant intermediate point \(\xi\) so that both segments \([a,\xi]\) and \([\xi,b]\) stay a positive distance from every obstacle.  This is the only step where \(N\ge2\) is essential: the set of forbidden directions is a finite union of small spherical caps and cannot exhaust the available directions.

Along each clear segment, apply a compactly supported smooth cylinder push that sends its initial endpoint to its terminal endpoint and is the identity near the cylinder boundary.  Because the support cylinder avoids all obstacles, the push fixes every point not currently moving.  Composing at most \(4K\) such pushes gives a compactly supported diffeomorphism satisfying \(H(z_j)=w_j\).  The explicit cylinder map, its inverse, the support verification, and the product--chain estimates leading to \eqref{eq:H-bound} are given in Appendix~\ref{app:finite-transitivity}.
\end{proof}

\subsubsection{Smooth autonomous reference realization}
\label{sec:smooth-reference-module}

\begin{proposition}[Smooth autonomous reference realization]
\label{prop:smooth-reference-module}
Let \(N\ge2\) and let \(\cD\) be admissible.  There exists a compactly supported smooth autonomous field \(f\) whose unit-time flow satisfies
\[
\Phi_f^1(x_i)=y_i,
\qquad i=1,\ldots,M.
\]
The marked trajectories remain pairwise distinct at every common time.  In the normalized setting
\[
\cV\subset\overline\B_{1/2},
\qquad \sep(\cV)=\rho,
\qquad 0<\rho\le1,
\]
the same construction admits support and fixed-order derivative bounds depending only on $N,K,\rho$; these bounds are recorded in Appendix~\ref{app:finite-transitivity}.
\end{proposition}

\begin{proof}
Choose an ordering \((\mathsf v_1,\ldots,\mathsf v_K)\) of the actual vertices and the corresponding ordering \((\mathsf v_1^0,\ldots,\mathsf v_K^0)\) of the standard vertices.  Lemma~\ref{lem:finite-transitivity} gives a compactly supported smooth diffeomorphism \(H\) satisfying
\begin{equation}
\label{eq:H-data-map}
H(\mathsf v_j^0)=\mathsf v_j\quad(j=1,\ldots,K),
\qquad
H(x_i^0)=x_i,\quad H(y_i^0)=y_i\quad(i=1,\ldots,M).
\end{equation}
Define the pushforward field \(f:=H_*f_0\), namely
\[
(H_*u)(x):=DH(H^{-1}x)u(H^{-1}x).
\]
Uniqueness of solutions gives the exact conjugacy
\begin{equation}
\label{eq:reference-conjugacy}
\Phi_f^t=H\circ\Phi_{f_0}^t\circ H^{-1}.
\end{equation}
Since every prescribed standard edge is traversed in unit time, \eqref{eq:H-data-map} and \eqref{eq:reference-conjugacy} yield \(\Phi_f^1(x_i)=y_i\).  Standard marked trajectories are either disjoint component circles or distinct synchronized positions on the same circle; conjugation by a diffeomorphism preserves this collision-free property.
\end{proof}

\subsection{Independent endpoint corrections}
\label{sec:endpoint-frame-module}

For the unit-time reference field, write \(\gamma_i(t):=\Phi_f^t(x_i)\), and let \(U_i(t,s)\) denote the fundamental matrix
\[
\partial_tU_i(t,s)=Df(\gamma_i(t))U_i(t,s),
\qquad U_i(s,s)=I_N.
\]
For a perturbation field \(h\), define the endpoint response
\[
\mathcal L_f(h)_i:=\int_0^1U_i(1,t)h(\gamma_i(t))\,\dd t,
\qquad
\mathcal L_f(h):=(\mathcal L_f(h)_1,\ldots,\mathcal L_f(h)_M).
\]
The usual variational equation gives \(D\End(f)[h]=\mathcal L_f(h)\).

\subsubsection{Private trajectory tubes}

For a nonstationary datum, let \(\gamma_i^0\) be its standard circular trajectory and let \(\nu(i)\) denote its completed component.  On \(I_{\rm tube}:=(1/4,3/4)\), introduce longitudinal--transverse coordinates
\begin{equation}
\Theta_i^0(t,z)
:=\gamma_i^0(t)
+z_1\bigl(\gamma_i^0(t)-c_{\nu(i)}\bigr)
+\sum_{a=2}^{N-1}z_ae_{a+1},
\qquad (t,z)\in I_{\rm tube}\times\R^{N-1}.
\end{equation}
Because only finitely many marked trajectories are involved, one may choose a common radius \(\varepsilon_{\rm tube}>0\) sufficiently small that the restriction
\[
\Theta_i^0:I_{\rm tube}\times\B_{2\varepsilon_{\rm tube}}^{N-1}\longrightarrow O_i^0
\]
is a diffeomorphism onto an open set \(O_i^0\), and the middle part of this tube meets no constrained standard trajectory except its own centerline:
\begin{equation}
\label{eq:standard-tube-privacy}
O_i^0\cap\gamma_j^0([0,1])=\varnothing\quad(j\ne i),
\qquad
O_i^0\cap\gamma_i^0([0,1])=\gamma_i^0(I_{\rm tube}).
\end{equation}
This is the geometric localization that permits one endpoint to be varied without affecting the others.  The injectivity, quantitative separation, and inverse-regularity estimates for \(\Theta_i^0\) are proved in Appendix~\ref{app:endpoint-coordinates}.

\subsubsection{Independent endpoint directions}

Choose smooth cutoffs \(\vartheta\in C_c^\infty(I_{\rm tube})\) and \(\zeta_i\in C_c^\infty(\B_{2\varepsilon_{\rm tube}}^{N-1})\) such that \(\int_0^1\vartheta(t)\,\dd t=1\) and \(\zeta_i(0)=1\).  Let \(U_i^0(t,s)\) be the fundamental matrix of \(f_0\) along \(\gamma_i^0\).  For \(a=1,\ldots,N\), define on \(O_i^0\)
\[
k_{i,a}^0(\Theta_i^0(t,z))
:=\vartheta(t)\zeta_i(z)U_i^0(1,t)^{-1}e_a,
\]
and extend it by zero.  By \eqref{eq:standard-tube-privacy}, this field vanishes along all constrained trajectories except the \(i\)th, while on its own centerline the fundamental matrix is exactly cancelled.  Therefore
\begin{equation}
\label{eq:standard-frame-response}
\int_0^1U_j^0(1,t)k_{i,a}^0(\gamma_j^0(t))\,\dd t
=\delta_{ij}e_a.
\end{equation}
For a self-loop datum, take a compactly supported bump equal to \(e_a\) near the stationary point, with support in a neighborhood where the standard field vanishes and which is disjoint from every other marked unit-time path; the same identity follows because the fundamental matrix there is the identity.

Transport and precondition these directions by setting
\begin{equation}
\label{eq:preconditioned-standard-frame}
h_{i,a}:=H_*\!\left(
\sum_{b=1}^N[DH(y_i^0)^{-1}]_{ba}k_{i,b}^0
\right).
\end{equation}
Differentiating the flow conjugacy gives
\begin{equation}
\label{eq:fundamental-conjugacy}
U_i(t,s)=DH(\gamma_i^0(t))U_i^0(t,s)DH(\gamma_i^0(s))^{-1}.
\end{equation}
Equations \eqref{eq:standard-frame-response}--\eqref{eq:fundamental-conjugacy} imply
\(\mathcal L_f(h_{i,a})_j=\delta_{ij}e_a\).  Relabeling the \(m=NM\) fields gives the central endpoint-frame identity
\begin{equation}
\label{eq:ideal-sensitivity-identity}
\bigl[\mathcal L_f(h_1)\ \cdots\ \mathcal L_f(h_m)\bigr]=I_m.
\end{equation}

\begin{proposition}[Independent endpoint corrections]
\label{prop:endpoint-frame-module}
The reference field from Proposition~\ref{prop:smooth-reference-module} admits compactly supported smooth directions \(h_1,\ldots,h_m\) satisfying \eqref{eq:ideal-sensitivity-identity}. In the normalized setting, their support and fixed-order $C^{s_N}$ bounds are recorded in Appendix~\ref{app:zero-extension}.
\end{proposition}

\begin{proof}
The endpoint identity has just been proved. Compact support follows from the private-tube construction and the compact support of \(H-\Id\). Smooth zero extension and the correction norms are verified in Appendix~\ref{app:zero-extension}.
\end{proof}

Figure~\ref{fig:multipoint-schematic} summarizes the private-tube localization and the resulting independent endpoint responses.

\begin{figure}[!ht]
\centering
\includegraphics[width=.8\textwidth]{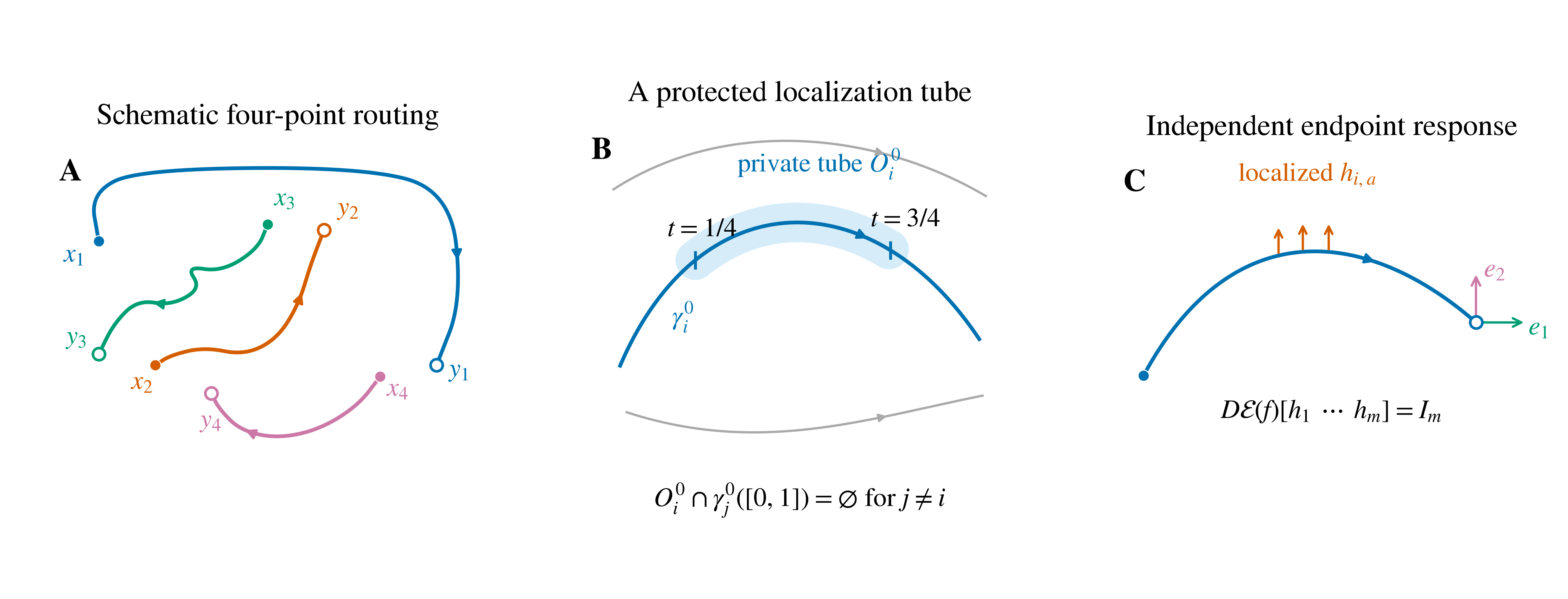}
\caption{Localized corrections for the smooth reference flow. (A) Marked reference paths join the prescribed inputs and targets. (B) In standard coordinates, a tube around the middle part of one moving path meets no other marked path during the unit time interval. (C) After transport and preconditioning, the correction fields give independent first-order endpoint changes and the identity response matrix in \eqref{eq:ideal-sensitivity-identity}; two response directions are shown schematically. Stationary samples use local bumps instead of tubes.}
\label{fig:multipoint-schematic}
\end{figure}

\subsection{Stability and Brouwer-degree correction}
\label{sec:stability-degree}

\subsubsection{Uniform endpoint stability}

\begin{lemma}[Uniform endpoint stability on a parameter ball]
\label{lem:uniform-endpoint-stability}
Let \(f,g_1,\ldots,g_m\) be locally Lipschitz and put
\(G_\alpha:=\sum_{\ell=1}^m\alpha_\ell g_\ell\).
Assume that, for \(|\alpha|\le r\), all trajectories of \(f+G_\alpha\) issued from \(x_1,\ldots,x_M\) exist on \([0,1]\) and remain in \(\overline\B_{R_0}\). Fix any \(R>R_0\). Then there are \(\eta_0>0\) and \(C>0\) such that, whenever \(V\) is locally Lipschitz, every field \(V+G_\alpha\), \(|\alpha|\le r\), has a global flow, and
\[
\norm{V-f}_{C(\overline\B_R)}\le\eta\le\eta_0,
\]
one has
\[
\sup_{|\alpha|\le r}
\abs{\End(V+G_\alpha)-\End(f+G_\alpha)}
\le C\eta.
\]
\end{lemma}

\begin{proof}
For each initial point and parameter \(|\alpha|\le r\), stop the trajectory of \(V+G_\alpha\) at its first exit from \(\B_R\). On the stopped interval, subtract the two integral equations. Since only finitely many locally Lipschitz fields \(g_\ell\) occur and \(|\alpha|\le r\), the family \(f+G_\alpha\) has a common Lipschitz constant on \(\overline\B_R\). Gronwall's inequality therefore gives
\[
\sup_{0\le t\le1}|z_{V,\alpha}(t)-z_{f,\alpha}(t)|
\le C_0\eta
\]
uniformly in \(\alpha\) and in the finitely many initial points, as long as the stopped trajectory is defined. Choose \(\eta_0\) so that \(C_0\eta_0<R-R_0\). Then a first exit from \(\B_R\) is impossible, so the estimate holds on the whole interval \([0,1]\). Evaluating at time one and stacking the \(M\) endpoint blocks proves the claim.
\end{proof}

\subsubsection{Brouwer-degree correction}

\begin{lemma}[Degree correction]
\label{lem:degree-exactification}
Let \(F:\overline\B_r^m\to\R^m\) be continuous, let \(J\in\R^{m\times m}\) be invertible, and set \(\sigma:=\min_{|v|=1}|Jv|\).  If
\[
\sup_{|\alpha|=r}|F(\alpha)-J\alpha|<\sigma r,
\]
then some \(\alpha_*\in\B_r^m\) satisfies \(F(\alpha_*)=0\).
\end{lemma}

\begin{proof}
For \(0\le\lambda\le1\), set
\[
\mathscr H_\lambda(\alpha)
:=J\alpha+\lambda\bigl(F(\alpha)-J\alpha\bigr).
\]
On \(|\alpha|=r\),
\[
|\mathscr H_\lambda(\alpha)|
\ge |J\alpha|-|F(\alpha)-J\alpha|>0.
\]
Thus the homotopy does not meet the origin on the boundary, and
\[
\deg(F,\B_r^m,0)
=\deg(J,\B_r^m,0)
=\operatorname{sgn}(\det J)\ne0.
\]
The existence property of Brouwer degree gives a zero of \(F\) in \(\B_r^m\).
\end{proof}

\subsubsection{The endpoint-correction principle}

\begin{lemma}[Local endpoint expansion]
\label{lem:endpoint-frame-linearization}
Suppose that
\[
\End(f)=Y,
\qquad
[\mathcal L_f(h_1)\ \cdots\ \mathcal L_f(h_m)]=I_m,
\]
with \(f,h_1,\ldots,h_m\in C_c^\infty(\R^N;\R^N)\).
Then there exist \(R>0\) and \(\varepsilon_0>0\) such that, whenever
\(g_1,\ldots,g_m\) are locally Lipschitz and
\[
\max_\ell \norm{g_\ell-h_\ell}_{C(\overline\B_R)}<\varepsilon_0,
\]
the matrix
\[
J:=[\mathcal L_f(g_1)\ \cdots\ \mathcal L_f(g_m)]
\]
is invertible. Writing \(G_\alpha:=\sum_{\ell=1}^m\alpha_\ell g_\ell\),
there exist \(r_0,C>0\) such that, for \(|\alpha|\le r_0\),
\[
\End(f+G_\alpha)=Y+J\alpha+R(\alpha),
\qquad
|R(\alpha)|\le C|\alpha|^2.
\]
\end{lemma}

\begin{proof}
	Choose \(R\) so that all reference trajectories lie in \(\B_{R-2}\).
	The map \(h\mapsto\mathcal L_f(h)\) is continuous in the uniform norm
	on \(\overline\B_R\), hence \(J\) remains invertible if
	\(g_\ell\) are sufficiently close to \(h_\ell\).
	
	Fix such \(g_\ell\), and set
	\[
	A_g:=\left(\sum_{\ell=1}^m
	\norm{g_\ell}_{C(\overline\B_R)}^2\right)^{1/2},
	\qquad
	L_g:=\left(\sum_{\ell=1}^m
	\Lip_{\overline\B_R}(g_\ell)^2\right)^{1/2}.
	\]
	Both constants are finite. Let \(z_i^\alpha\) be the trajectory of
	\(f+G_\alpha\) issued from \(x_i\), stopped at its first exit from
	\(\B_{R-1}\), and write
	\(e_i^\alpha:=z_i^\alpha-\gamma_i\). On the stopped interval,
	Cauchy--Schwarz gives
	\[
	|G_\alpha(z)|\le |\alpha|A_g,\qquad
	|G_\alpha(z)-G_\alpha(z')|
	\le |\alpha|L_g|z-z'|.
	\]
	Thus, with \(L_f:=\Lip_{\overline\B_R}(f)\), the integral equations
	and Gronwall's inequality give
	\[
	\norm{e_i^\alpha}_{C([0,t])}
	\le A_g e^{L_f}|\alpha|,
	\qquad 0\le t\le1.
	\]
	After decreasing \(r_0\), the right-hand side is \(<1\), so the
	stopped trajectory cannot reach \(\partial\B_{R-1}\). Consequently
	every \(z_i^\alpha\) exists on \([0,1]\), remains in
	\(\B_{R-1}\), and satisfies the displayed estimate.

	Let \(\zeta_i^\alpha\) solve
	\[
	\dot\zeta_i^\alpha
=
Df(\gamma_i)\zeta_i^\alpha+G_\alpha(\gamma_i),
	\qquad
	\zeta_i^\alpha(0)=0.
	\]
	Variation of constants first gives
	\(\norm{\zeta_i^\alpha}_{C([0,1])}\le C_1|\alpha|\).
	For
	\(w_i^\alpha:=e_i^\alpha-\zeta_i^\alpha\), subtraction of the two
	equations yields
	\[
	\dot w_i^\alpha
	=Df(\gamma_i)w_i^\alpha
	+\bigl[f(\gamma_i+e_i^\alpha)-f(\gamma_i)
	-Df(\gamma_i)e_i^\alpha\bigr]
	+\bigl[G_\alpha(\gamma_i+e_i^\alpha)-G_\alpha(\gamma_i)\bigr].
	\]
	Since \(f\) is \(C^2\) on \(\overline\B_R\), Taylor's formula and
	the preceding bounds imply, uniformly in \(i\) and \(t\),
	\[
	\left|f(\gamma_i+e_i^\alpha)-f(\gamma_i)
	-Df(\gamma_i)e_i^\alpha\right|
	\le \tfrac12\norm{D^2f}_{C(\overline\B_R)}
	|e_i^\alpha|^2=O(|\alpha|^2),
	\]
	and
	\[
	|G_\alpha(\gamma_i+e_i^\alpha)-G_\alpha(\gamma_i)|
	\le |\alpha|L_g|e_i^\alpha|=O(|\alpha|^2).
	\]
	A second application of Gronwall's inequality therefore gives
	\[
	\norm{z_i^\alpha-\gamma_i-\zeta_i^\alpha}_{C([0,1])}
	=\norm{w_i^\alpha}_{C([0,1])}
	\le C_2|\alpha|^2.
	\]
By variation of constants,
\[
\zeta_i^\alpha(1)
=
\sum_{\ell=1}^m
\alpha_\ell\,\mathcal L_f(g_\ell)_i.
\]
Stacking the endpoint blocks and using \(\End(f)=Y\) yields
\[
\End(f+G_\alpha)=Y+J\alpha+O(|\alpha|^2),
\]
which proves the claim.
\end{proof}

\begin{proposition}[Robust endpoint correction]
\label{prop:robust-exactification}
Under the assumptions of Lemma~\ref{lem:endpoint-frame-linearization},
there exist \(R>0\) and \(\varepsilon_0>0\) such that, for every family
\(g_1,\ldots,g_m\) satisfying
\[
\max_\ell\norm{g_\ell-h_\ell}_{C(\overline\B_R)}<\varepsilon_0,
\]
there are \(r,\eta_0>0\) with the following property: if
\[
\norm{V_{\rm base}-f}_{C(\overline\B_R)}<\eta_0
\]
and every
\[
V_{\rm base}+\sum_{\ell=1}^m\alpha_\ell g_\ell,
\qquad |\alpha|\le r,
\]
has a global flow, then some \(\alpha_*\in\B_r^m\) satisfies
\[
\End\!\left(
V_{\rm base}+\sum_{\ell=1}^m\alpha_{*,\ell}g_\ell
\right)=Y.
\]
\end{proposition}

The robustness statement uses two continuity facts.  First, the endpoint-response operator $g\mapsto\mathcal L_f(g)$ is continuous with respect to uniform perturbations on $\overline\B_R$, so sufficiently accurate approximations of the directions $h_\ell$ preserve invertibility of the endpoint-response matrix.  Second, once $g_1,\ldots,g_m$ are fixed, Lemma~\ref{lem:uniform-endpoint-stability} gives uniform continuous dependence of $\End(V_{\rm base}+G_\alpha)$ on the base field for $|\alpha|\le r$.  Consequently, the admissible radii and tolerances may depend on the chosen approximate frame, in particular through its local Lipschitz bounds, but no derivative approximation of the base field is required.

\begin{proof}
Fix \(g_1,\ldots,g_m\) and let
\[
J=[\mathcal L_f(g_1)\ \cdots\ \mathcal L_f(g_m)],
\qquad
\sigma:=\min_{|v|=1}|Jv|>0.
\]
By Lemma~\ref{lem:endpoint-frame-linearization}, after choosing \(r>0\)
small enough,
\[
|\End(f+G_\alpha)-Y-J\alpha|
\le \frac{\sigma}{4}|\alpha|,
\qquad |\alpha|\le r.
\]
Lemma~\ref{lem:uniform-endpoint-stability} then gives
\(\eta_0>0\) such that
\[
\sup_{|\alpha|\le r}
|\End(V_{\rm base}+G_\alpha)-\End(f+G_\alpha)|
\le \frac{\sigma r}{4}
\]
whenever
\(\norm{V_{\rm base}-f}_{C(\overline\B_R)}<\eta_0\).

Set
\[
F(\alpha):=\End(V_{\rm base}+G_\alpha)-Y.
\]
For the fixed correction fields and base field, \(F\) is continuous on the closed parameter ball: the vector fields depend continuously on the coefficients, have a common local Lipschitz bound there, and their relevant trajectories remain in the common trapping ball.
For \(|\alpha|=r\),
\[
|F(\alpha)-J\alpha|
\le \frac{\sigma r}{2}
<\sigma r.
\]
Lemma~\ref{lem:degree-exactification} therefore yields
\(\alpha_*\in\B_r^m\) with \(F(\alpha_*)=0\), proving the result.
\end{proof}

Figure~\ref{fig:proof-synopsis} illustrates the residual-enclosure argument underlying the degree step.

\begin{figure}[!ht]
\centering
\includegraphics[width=.8\textwidth]{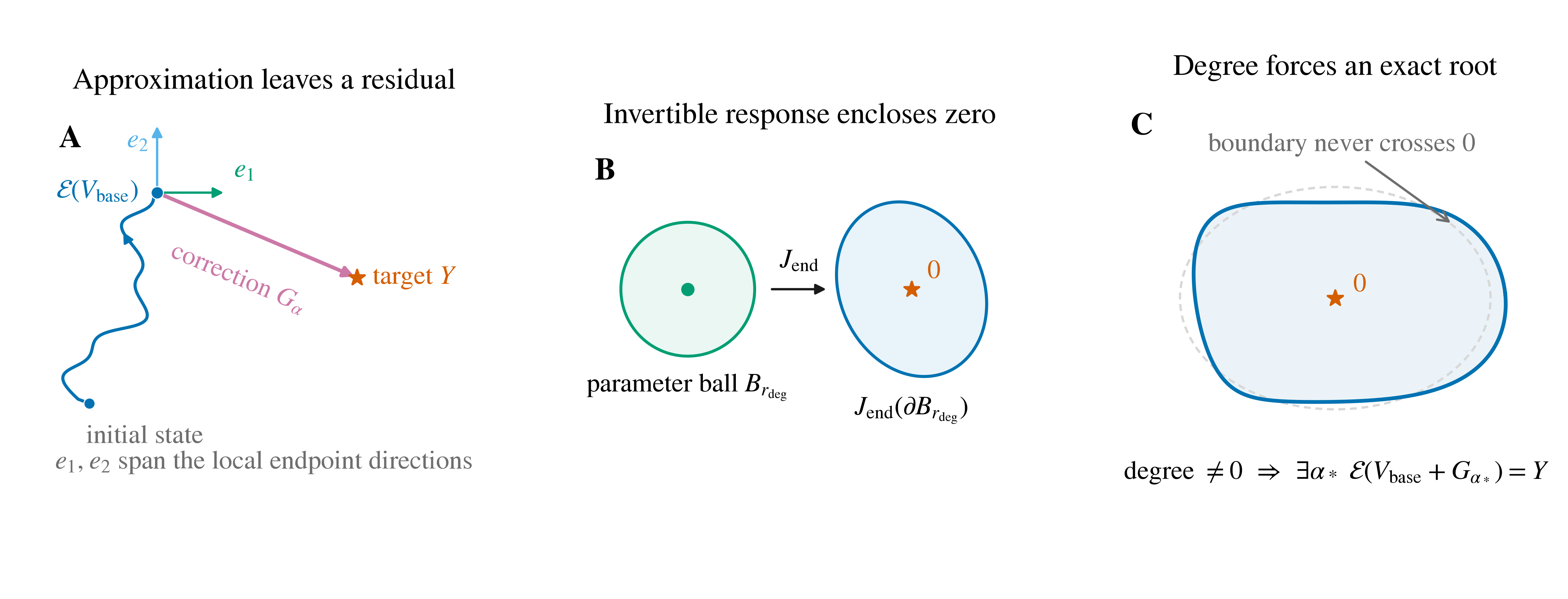}
\caption{The degree argument in endpoint space, shown in two dimensions. Let $F(\alpha)$ be the stacked endpoint error after adding the correction with coefficients $\alpha$. The linear response $J$ maps the parameter sphere to a boundary separated from zero. The estimate $\|F(\alpha)-J\alpha\|<\sigma r$ on $|\alpha|=r$, where $\sigma$ is the smallest singular value of $J$, keeps the straight homotopy between $F$ and $J$ away from zero on that boundary. Since $J$ is invertible, the degree is nonzero and $F$ vanishes for a coefficient vector inside the ball.}
\label{fig:proof-synopsis}
\end{figure}

\subsection{Proof of the approximation-to-interpolation theorem}
\label{sec:proof-main-theorem}

\begin{proof}[Proof of Theorem~\ref{thm:abstract-exactification}]
\textbf{Step 1: reference field.} It is enough to consider \(T=1\), since linearity of \(\mathscr F\) implies that \(V\in\mathscr F\) entails \(T^{-1}V\in\mathscr F\), and \(\Phi_{T^{-1}V}^{T}=\Phi_V^1\). Proposition~\ref{prop:smooth-reference-module} provides a compactly supported smooth field \(f\) with \(\End(f)=Y\).

\textbf{Step 2: correction fields.} Proposition~\ref{prop:endpoint-frame-module} provides compactly supported smooth fields \(h_1,\ldots,h_m\) with identity first-order endpoint response. Proposition~\ref{prop:robust-exactification} supplies a correction tolerance \(\varepsilon_0\) on a ball \(\overline\B_R\).

\textbf{Step 3: approximation in the prescribed class.} First use the approximation-on-bounded-sets hypothesis to choose \(g_1,\ldots,g_m\in\mathscr F\) within this tolerance. For these fixed approximate correction fields, Proposition~\ref{prop:robust-exactification} supplies \(r>0\) and \(\eta_0>0\). Use the same approximation property once more to choose \(V_{\rm base}\in\mathscr F\) with
\[
\norm{V_{\rm base}-f}_{C(\overline\B_R)}<\eta_0.
\]
\textbf{Step 4: exact correction.} Since \(\mathscr F\) is linear, every field
\[
V_{\rm base}+\sum_{\ell=1}^m\alpha_\ell g_\ell,
\qquad |\alpha|\le r,
\]
belongs to \(\mathscr F\), and by assumption it has a global flow. Proposition~\ref{prop:robust-exactification} therefore yields a parameter \(\alpha_*\) realizing all endpoints exactly. Rescaling time gives the conclusion for arbitrary \(T>0\).
\end{proof}

\begin{remark}[Qualitative exact interpolation beyond ReLU]

Theorem~\ref{thm:abstract-exactification} isolates the qualitative core of the positive result. It uses only linear closure, local Lipschitz regularity, global well-posedness of the candidate flows, and approximation of compactly supported smooth fields on bounded sets. In particular, neither piecewise linearity nor positive homogeneity of ReLU is required for exact finite interpolation itself; see also Remark~\ref{rem:beyond-relu} for the corresponding extension to other activation functions.

{\hbadness=10000\sloppy
	This separates two levels of the positive theory. Qualitative exact interpolation follows from approximation on bounded sets together with a finite-dimensional endpoint correction and Brouwer degree. By contrast, the quantitative width and normalized-strength estimates in Theorem~\ref{thm:quantitative-exactification} use specifically quantitative shallow-network approximation estimates and the positive homogeneity of ReLU. \par}
\end{remark}

\FloatBarrier
\section{Shallow ReLU realizations: qualitative and quantitative consequences}
\label{sec:quantitative}

\subsection{Qualitative consequence of the abstract theorem}
\label{sec:relu-qualitative-application}

\begin{proof}[Proof of Corollary~\ref{thm:qualitative-interpolation}]
The class $\NN_N$ is linear, every one of its fields is globally Lipschitz and hence globally well posed, and classical universal approximation for continuous nonpolynomial ridge activations gives the required coordinatewise approximation on every compact ball \cite{hornik1991,pinkus1999}. In particular, each compactly supported smooth vector field can be approximated there by a finite vector-valued ReLU ridge expansion. Theorem~\ref{thm:abstract-exactification} therefore applies. Positive time rescaling preserves $\NN_N$ by scalar closure.
\end{proof}

\subsection{Quantitative preliminaries}

Recall the fixed regularity order $s_N$ from \eqref{eq:fixed-regularity-order}.  The ReLU approximation rate uses the additional dimension-dependent exponent
\begin{equation}
\label{eq:relu-approx-exponent}
\kappa_N:=\frac{2N}{N+3}.
\end{equation}
Unlike $s_N$, which indexes the fixed-order regularity of the geometric construction, $\kappa_N$ enters only in the quantitative approximation argument below.

\subsubsection{Parameter conventions and normalized strength}

For a parameter list \(\theta=((w_j,a_j,b_j))_{j=1}^p\), write
\begin{equation}
\Str(\theta):=\sum_{j=1}^p\abs{w_j}\sqrt{\abs{a_j}^2+b_j^2},
\end{equation}
and, when a shallow field is equipped with this particular list, also write \(p(V_\theta):=p(\theta)\) and \(\Str(V_\theta):=\Str(\theta)\).  These are representation-dependent quantities.  The affine-normalized versions \(\Str_S(\theta)\) and \(\norm{\theta}_{S,2}\) were introduced in the quantitative refinement of Section~\ref{sec:results} solely to state the coordinate-invariant quantitative theorem.

Let \(S(x)=B(x-x_*)\) and let
$
\widetilde V(z):=BV(B^{-1}z+x_*)
$
be the conjugated field.  Directly from the shallow representation,
\begin{equation}
\label{eq:functional-strength}
\max\{\Lip(\widetilde V),|\widetilde V(0)|\}
\le\Str_S(\theta).
\end{equation}
Thus \(T\Str_S(\theta)\) is a genuine dynamical resource: it controls the logarithmic bi-Lipschitz distortion and the displacement generated by the normalized flow over time \(T\), not merely the number of parameters.

Positive homogeneity of ReLU also allows each unit to be balanced independently.  After zero units are discarded,
\begin{equation}
\label{eq:balancing}
\inf_{\lambda_1,\ldots,\lambda_p>0}
\norm{
((\lambda_j^{-1}w_j,\lambda_ja_j,\lambda_jb_j))_{j=1}^p
}_{S,2}^2
=2\Str_S(\theta).
\end{equation}
Consequently, \(2\Str_S(\theta)\) is the least squared normalized Euclidean parameter norm among unitwise positive-homogeneity rescalings representing the same vector field.  This separates the functional size of the realized vector field from an arbitrary choice of unit scaling.

\begin{remark}
The representation-dependent quantity \(\Str_S(\theta)\) therefore has two roles: it controls the conjugated vector field through \eqref{eq:functional-strength}, and after balancing it controls the normalized Euclidean parameter size through \eqref{eq:balancing}.  These facts are needed only for the quantitative ReLU realization and for the bounded-strength obstruction; they play no role in the qualitative density-to-exactness theorem.
\end{remark}

The coefficient-controlled ridge approximation below uses the exponent \(\kappa_N\) from \eqref{eq:relu-approx-exponent}, while the fixed order \(s_N\) from \eqref{eq:fixed-regularity-order} indexes the geometric derivative bounds in Appendix~\ref{app:finite-transitivity} and the quantitative endpoint-frame bridge.

\subsubsection{Coefficient-controlled ReLU approximation}

\begin{lemma}[Coefficient-controlled approximation of smooth vector fields]
\label{lem:smooth-vector-approximation}
Let \(s>(N+3)/2\) be an integer, let \(R\ge1\), and let \(u\in C_c^s(\R^N;\R^N)\). For every \(0<\varepsilon_{\rm app}\le1\), set
\begin{equation}
\label{eq:explicit-vector-budgets}
\begin{aligned}
\mathcal B_{R,s}(u)
&:=2C_{\mathrm{emb},N,s}(1+R)^s\norm u_{C^s},\\
n_{\rm app}(R,s;u,\varepsilon_{\rm app})
&:=
\begin{cases}
0,&u=0,\\[1mm]
\displaystyle
\left\lceil
\left(\frac{A_N\sqrt N\,\mathcal B_{R,s}(u)}{\varepsilon_{\rm app}}\right)^{\kappa_N}
\right\rceil,&u\ne0.
\end{cases}
\end{aligned}
\end{equation}
and
\begin{equation}
\label{eq:explicit-vector-width-strength}
\mathcal P_{R,s}(u,\varepsilon_{\rm app}):=Nn_{\rm app}(R,s;u,\varepsilon_{\rm app}),
\qquad
\mathcal S_{R,s}(u):=N\sqrt{1+R^{-2}}\,\mathcal B_{R,s}(u).
\end{equation}
Then there is a shallow ReLU field \(V_{\rm app}\) such that
\[
\norm{V_{\rm app}-u}_{C(\overline\B_R)}\le\varepsilon_{\rm app},
\qquad
p(V_{\rm app})\le\mathcal P_{R,s}(u,\varepsilon_{\rm app}),
\qquad
\Str(V_{\rm app})\le\mathcal S_{R,s}(u).
\]
\end{lemma}

This ReLU-specific approximation input is proved in Appendix~\ref{app:relu-approximation}.

\subsection{A quantitative endpoint-correction bridge}

The following proposition quantifies the compact region and tolerances left implicit in Proposition~\ref{prop:robust-exactification}; its proof is in Appendix~\ref{app:stability-degree}.

For a family \(\mathbf g=(g_1,\ldots,g_m)\) of globally Lipschitz vector fields, define the correction-family budget
\begin{equation}
\mathfrak B(\mathbf g)
:=
\left(
\sum_{\ell=1}^m
\bigl(|g_\ell(0)|^2+\Lip(g_\ell)^2\bigr)
\right)^{1/2}.
\end{equation}

\begin{proposition}[Quantitative endpoint-correction bridge]
\label{prop:quantitative-interface}
For every \(N\ge2\) there is \(C_N\ge2\), depending only on \(N\), with the following property.  Let \(\cD\) be an admissible unit-time dataset with
\[
\cV\subset\overline\B_{1/2},
\qquad
K:=|\cV|\ge2,
\qquad
\rho:=\sep(\cV)\in(0,1],
\qquad
m:=NM.
\]
The reference field and endpoint correction directions in Propositions~\ref{prop:smooth-reference-module} and~\ref{prop:endpoint-frame-module} can be chosen with \(R_{\rm tr}=8K+6\) so that
\[
\End(f)=Y,
\qquad
[\mathcal L_f(h_1)\ \cdots\ \mathcal L_f(h_m)]=I_m,
\qquad
\supp f\cup\bigcup_{\ell=1}^m\supp h_\ell\subset\B_{R_{\rm tr}}.
\]
There also exist an active approximation radius \(R_*\ge R_{\rm tr}\) and a frame tolerance \(0<\varepsilon_*\le1\), depending only on the normalized dataset, such that, with
\begin{equation}
\Gamma_N(K,\rho):=
\left(C_N\frac K\rho\right)^{C_NK},
\end{equation}
one has
\begin{equation}
\label{eq:bridge-geometry-outputs}
\max\left\{
R_*,\ \varepsilon_*^{-1},\
\norm f_{C^{s_N}},\
\max_{1\le\ell\le m}\norm{h_\ell}_{C^{s_N}}
\right\}
\le \Gamma_N(K,\rho).
\end{equation}

For every scalar budget \(\mathfrak b\ge1\), there are
\[
0<r_*=r_*(\cD,\mathfrak b)\le\frac18,
\qquad
0<\eta_*=\eta_*(\cD,\mathfrak b)\le1,
\]
satisfying
\begin{equation}
\label{eq:bridge-budget-envelope}
\max\{r_*^{-1},\eta_*^{-1}\}
\le
\Gamma_N(K,\rho)(1+\mathfrak b)^2,
\end{equation}
with the following property.  If globally Lipschitz fields \(g_1,\ldots,g_m\) and a locally Lipschitz field \(V_{\rm base}\) satisfy
\begin{equation}
\label{eq:bridge-hypotheses}
\max_{1\le\ell\le m}
\norm{g_\ell-h_\ell}_{C(\overline\B_{R_{\rm tr}})}
\le\varepsilon_*,
\qquad
\mathfrak B(\mathbf g)\le\mathfrak b,
\qquad
\norm{V_{\rm base}-f}_{C(\overline\B_{R_*})}\le\eta_*,
\end{equation}
and every field
\(
V_{\rm base}+\sum_{\ell=1}^m\alpha_\ell g_\ell
\), \(|\alpha|\le r_*\), has a global flow, then there exists \(\alpha_*\in\B_{r_*}^m\) such that
\begin{equation}
\label{eq:bridge-exactness}
\End\!\Big(V_{\rm base}+\sum_{\ell=1}^m\alpha_{*,\ell}g_\ell\Big)=Y.
\end{equation}
Moreover, every sample trajectory entering the endpoint homotopy for \(|\alpha|\le r_*\), both for \(f+\sum_\ell\alpha_\ell g_\ell\) and for \(V_{\rm base}+\sum_\ell\alpha_\ell g_\ell\), remains in \(\B_{R_*}\) for \(0\le t\le1\).
\end{proposition}

\begin{proof}
See Appendix~\ref{app:stability-degree}.
\end{proof}

\subsection{Controlled shallow-ReLU realization in normalized coordinates}

The affine aspect ratio defined in Section~\ref{sec:results} removes overall scale and anisotropic coordinate distortion before the configuration is routed.  Assume first that \(K\ge2\).  By the definition of the infimum \(\mathfrak a(\cV)\), choose a \(2\)-near minimizer \(L_{\rm aff}\in\mathrm{GL}(N)\) such that
\[
\frac{\diam(L_{\rm aff}\cV)}{\sep(L_{\rm aff}\cV)}
\le 2\mathfrak a(\cV).
\]
Fix \(x_*\in\cV\) and put
\[
B:=\frac{L_{\rm aff}}{2\diam(L_{\rm aff}\cV)},
\qquad
S(x):=B(x-x_*).
\]
Then \(S(\cV)\subset\overline\B_{1/2}\), and with \(\rho:=\sep(S\cV)\),
\begin{equation}
\label{eq:rho-aspect}
\rho
=\frac{\sep(L_{\rm aff}\cV)}{2\diam(L_{\rm aff}\cV)}
\ge\frac{1}{4\mathfrak a(\cV)}.
\end{equation}
Work in these normalized coordinates and at unit time.  Take the reference field \(f\), endpoint directions \(h_1,\ldots,h_m\), radius \(R_*\), and frame tolerance \(\varepsilon_*\) supplied by Proposition~\ref{prop:quantitative-interface}.

Define the aggregate ReLU strength envelope of the ideal endpoint directions by
\begin{equation}
\mathcal S_h
:=\left(
\sum_{\ell=1}^m
\mathcal S_{R_{\rm tr},s_N}(h_\ell)^2
\right)^{1/2},
\qquad
\mathfrak b_h:=\max\{1,\sqrt2\,\mathcal S_h\}.
\end{equation}
Use the budget-dependent part of Proposition~\ref{prop:quantitative-interface} with \(\mathfrak b=\mathfrak b_h\), obtaining \(r_*\) and \(\eta_*\).  Now set
\begin{equation}
\mathcal P_h
:=\sum_{\ell=1}^m
\mathcal P_{R_{\rm tr},s_N}(h_\ell,\varepsilon_*),
\qquad
\mathcal P_f
:=\mathcal P_{R_*,s_N}(f,\eta_*),
\qquad
\mathcal S_f
:=\mathcal S_{R_*,s_N}(f).
\end{equation}
Lemma~\ref{lem:smooth-vector-approximation} gives shallow ReLU fields \(g_1,\ldots,g_m\) such that
\[
\max_\ell\norm{g_\ell-h_\ell}_{C(\overline\B_{R_{\rm tr}})}\le\varepsilon_*,
\qquad
\Str(g_\ell)\le\mathcal S_{R_{\rm tr},s_N}(h_\ell),
\]
and a shallow ReLU field \(V_{\rm base}\) such that
\[
\norm{V_{\rm base}-f}_{C(\overline\B_{R_*})}\le\eta_*,
\qquad
p(V_{\rm base})\le\mathcal P_f,
\qquad
\Str(V_{\rm base})\le\mathcal S_f.
\]
For each shallow ReLU field \(g\), equation~\eqref{eq:functional-strength} with the identity normalization gives \(|g(0)|,\Lip(g)\le\Str(g)\).  Hence
\begin{equation}
\mathfrak B(\mathbf g)^2
\le2\sum_{\ell=1}^m\Str(g_\ell)^2
\le2\mathcal S_h^2
\le\mathfrak b_h^2.
\end{equation}
Every shallow ReLU field is globally Lipschitz, so all flow hypotheses in Proposition~\ref{prop:quantitative-interface} are automatic.  The bridge therefore gives \(\alpha_*\in\B_{r_*}^m\) for which, with \(G_\alpha:=\sum_{\ell=1}^m\alpha_\ell g_\ell\),
\begin{equation}
\label{eq:normalized-exact-field}
\End(V_{\rm base}+G_{\alpha_*})=Y.
\end{equation}
Thus \(\widehat V:=V_{\rm base}+G_{\alpha_*}\) is one shallow autonomous ReLU field realizing all normalized endpoints exactly.

\subsection{Width--strength accounting and proof of the quantitative theorem}

Concatenate the parameter lists of \(V_{\rm base}\) and the scaled fields \(\alpha_{*,\ell}g_\ell\), and call the resulting list \(\widehat\theta\).  Then
\begin{equation}
\label{eq:exact-normalized-budgets}
\widehat p:=p(\widehat\theta)\le\mathcal P_f+\mathcal P_h,
\qquad
\Str(\widehat\theta)\le\mathcal S_f+r_*\mathcal S_h.
\end{equation}
Appendix~\ref{app:quantitative-closure} closes these quantities under one coarse envelope.  Hence, after enlarging a dimension-only constant \(C_N\),\vspace{-2mm}
\begin{equation}
\label{eq:normalized-joint}
\widehat p,\ \Str(\widehat\theta)
\le\left(C_N\frac K\rho\right)^{C_NK}.\vspace{-2mm}
\end{equation}

\begin{proof}[Proof of Theorem~\ref{thm:quantitative-exactification}]
For \(K=1\), admissibility forces the unique sample to be stationary and the empty parameter list suffices.  Suppose \(K\ge2\) and use the normalized realization above.  Write
\(\widehat V(z)=\sum_{j=1}^{\widehat p}\widehat w_j(\widehat a_j^\top z+\widehat b_j)_+\), and define
\begin{equation}
\label{eq:pullback-parameters}
w_j:=\frac1T B^{-1}\widehat w_j,
\qquad
a_j:=B^\top\widehat a_j,
\qquad
b_j:=\widehat b_j-a_j^\top x_*.
\end{equation}
Then \(V_\theta(x)=T^{-1}B^{-1}\widehat V(Sx)\), and uniqueness of solutions gives
\(
S\circ\Phi_{V_\theta}^t
=\Phi_{\widehat V}^{t/T}\circ S.
\)
At \(t=T\), equation \eqref{eq:normalized-exact-field} yields the original endpoint constraints.  The affine pullback of each ridge unit is again one ridge unit, so the width is unchanged, while
\(T\Str_S(\theta)=\Str(\widehat\theta)\).  Combining \eqref{eq:normalized-joint} with \eqref{eq:rho-aspect} gives
\[
p(\theta),\ T\Str_S(\theta)
\le
\bigl(C_NK\mathfrak a(\cV)\bigr)^{C_NK}
=\mathfrak U_N(\cD).
\]
Finally, the positive-homogeneity balancing identity \eqref{eq:balancing} gives \eqref{eq:main-balanced}.  This proves \eqref{eq:main-exactness}--\eqref{eq:main-affine}.
\end{proof}

\subsection{Interpretation, coordinate bounds, and examples}

For a directly computable estimate in the original coordinates, define
\[
\Gamma_{\mathrm{coord}}(\cV)
:=
\begin{cases}
1,&K=1,\\[2mm]
\displaystyle
\frac{1+\max_{z\in\cV}\abs z}{\min\{1,\sep(\cV)\}},&K\ge2.
\end{cases}
\]

\begin{corollary}[Coordinate width--strength and individual-parameter bound]
\label{cor:coordinate-bound}
Let \(N\ge2\) and \(T>0\).  The constant \(C_N\) in Theorem~\ref{thm:quantitative-exactification} may be chosen so that every admissible dataset admits a parameter list \(\theta\) satisfying \(\Phi_{V_\theta}^T(x_i)=y_i\) for \(i=1,\ldots,M\) and
\begin{equation}
\label{eq:main-coordinate}
p(\theta),\ T\Str(\theta)
\le
\bigl(C_NM\Gamma_{\mathrm{coord}}(\cV)\bigr)^{C_NM}.
\end{equation}
Writing \(W_\theta=[w_1\ \cdots\ w_p]\), taking the rows of \(A_\theta\) to be \(a_j^\top\), and writing \(b_\theta=(b_j)_{j=1}^p\), the units may moreover be balanced so that
\begin{equation}
\label{eq:coordinate-parameter-bound}
\left(\norm{W_\theta}_{\mathrm F}^2+
\norm{A_\theta}_{\mathrm F}^2+
\norm{b_\theta}_2^2\right)^{1/2}
\le
\sqrt{\frac2T}\,
\bigl(C_NM\Gamma_{\mathrm{coord}}(\cV)\bigr)^{C_NM/2}.
\end{equation}
In particular, the same right-hand side bounds every individual scalar parameter and also the operator norms of the two weight matrices.
\end{corollary}

The proof of Corollary~\ref{cor:coordinate-bound} is given in Appendix~\ref{app:quantitative-closure}.

\begin{center}
\fbox{\begin{minipage}{0.93\textwidth}
\textbf{An explicit width-four example.}

Take $N=2$, $T=1$, and the four inputs $x_1=(1,0)$, $x_2=(-1,0)$, $x_3=(0,2)$, and $x_4=(0,-2)$, with targets $y_i=-x_i$. Define four ReLU units with zero biases by
\[
\begin{array}{c|cccc}
	j&1&2&3&4\\ \hline
	a_j&(1,0)&(-1,0)&(0,1)&(0,-1)\\
	w_j&(0,\pi)&(0,-\pi)&(-\pi,0)&(\pi,0).
\end{array}
\]
Since $s=s_+-(-s)_+$, their sum is exactly
\[
V(x)=\sum_{j=1}^4w_j(a_j^\top x)_+=\pi(-x_2,x_1).
\]
Hence $\Phi_V^t(x)=R_{\pi t}x$, where $R_{\pi t}$ is planar rotation through angle $\pi t$, and $\Phi_V^1(x_i)=-x_i=y_i$ for all four points. The field has width $p=4$ and raw strength $\Str(\theta)=\sum_j|w_j||a_j|=4\pi$. Figure~\ref{fig:dimension-routing} samples these four exact trajectories from the analytic formula. This small instance is far cheaper than the deliberately coarse worst-case envelope in Theorem~\ref{thm:quantitative-exactification}.
\end{minipage}}
\end{center}

\paragraph{Interpretation of the quantitative envelope.}
The coarse exponential dependence on \(K\) comes from protected finite-point relocation and derivative propagation through \(H\); the remaining losses come from meeting the bridge tolerances.  Optimizing the scaling is open.

\begin{remark}[Time--strength scaling and the dimension-only constant]
Let \(\widetilde V(z):=BV_\theta(B^{-1}z+x_*)\) and \(\widetilde F:=\Phi_{\widetilde V}^T\).  By \eqref{eq:functional-strength},
\(
\Lip(\widetilde V),|\widetilde V(0)|\le\Str_S(\theta).
\)
Gr\"onwall's inequality, applied to the forward and inverse flows, gives
\[
e^{-T\Str_S(\theta)}|z-z'|
\le |\widetilde F(z)-\widetilde F(z')|
\le e^{T\Str_S(\theta)}|z-z'|,
\]
so \(T\Str_S(\theta)\) controls the logarithmic bi-Lipschitz distortion.  The same argument with the growth bound \(|\widetilde V(z)|\le\Str_S(\theta)(1+|z|)\) yields
\[
|\widetilde F(z)-z|
\le(1+|z|)\bigl(e^{T\Str_S(\theta)}-1\bigr).
\]
Hence any exact interpolant satisfies, for every nonstationary datum,
\[
T\Str_S(\theta)
\ge
\log\!\left(1+\frac{|Sy_i-Sx_i|}{1+|Sx_i|}\right).
\]
Under the normalization \eqref{eq:main-normalization}, \(|Sx_i|\le1/2\) and \(|Sy_i-Sx_i|\ge[4\mathfrak a(\mathcal V)]^{-1}\), so
\[
T\Str_S(\theta)
\ge
\log\!\left(1+\frac{1}{6\mathfrak a(\mathcal V)}\right).
\]
Thus the inverse-\(T\) scaling in Theorem~\ref{thm:quantitative-exactification} reflects a genuine time--strength tradeoff rather than an artifact of the construction.

Here \(C_N\) depends only on \(N\) and absorbs fixed cutoff, approximation, and finite-order calculus constants; all data dependence remains explicit through \(K\) and \(\mathfrak a(\mathcal V)\).  For the normalized unit-time list \(\widehat\theta\), the pullback satisfies \(T\Str_S(\theta)=\Str(\widehat\theta)\).
\end{remark}

\begin{remark}[Two elementary exact width values]
If all samples are stationary, then \(p_{\min}(\cD,T)=0\).  If \(M=1\) and \(x_1\ne y_1\), the one-unit field \(V(x)=T^{-1}(y_1-x_1)(0^\top x+1)_+\) realizes the datum exactly, and hence \(p_{\min}(\cD,T)=1\). Indeed, \(p=0\) gives only the zero vector field and therefore cannot move \(x_1\) to the distinct point \(y_1\).
\end{remark}

\FloatBarrier
\section{Structural and metric obstructions}
\label{sec:cell-limitations}

This section collects two complementary limitations of autonomous flows.  In one dimension, universal exact point interpolation already fails by order preservation; see Remark~\ref{rem:one-dimensional-obstruction}.  In dimensions \(N\ge2\), exact point interpolation is possible, but autonomy still imposes a structural obstruction to universally quantified fixed-radius neighborhood routing, as expressed by Theorem~\ref{thm:two-cell}.  We first prove this semigroup obstruction and its sharp point-to-neighborhood consequence, and then record the distinct metric obstruction caused by bounded time--strength.

\subsection{Semigroup obstruction to neighborhood routing}

The pointwise exact-interpolation theory does not extend to the universally quantified fixed-radius property.  The following proof uses only continuity, autonomy, and the semigroup identity; it is independent of network width or approximation accuracy.

\begin{proof}[Proof of Theorem~\ref{thm:two-cell}]
Let \(C_a:=\overline\B_\delta(a)\), \(C_b:=\overline\B_\delta(b)\), and suppose, to the contrary, that \(\Psi^T(C_a)\subset\B_\varepsilon(b)\) and \(\Psi^T(C_b)\subset\B_\varepsilon(a)\). Because \(\varepsilon<\delta\), \(\B_\varepsilon(b)\subset C_b\). The semigroup identity gives
\begin{equation}
\label{eq:return-compression}
\Psi^{2T}(C_a)
=\Psi^T(\Psi^T(C_a))
\subset\Psi^T(C_b)
\subset\B_\varepsilon(a).
\end{equation}
Choose \(\rho\) with \(\varepsilon<\rho<\delta\) and set \(K_a:=\overline\B_\rho(a)\). Since \(K_a\subset C_a\), equation \eqref{eq:return-compression} implies
\[
\Psi^{2T}(K_a)\subset\B_\varepsilon(a)\subset K_a.
\]
Thus \(\Psi^{2T}|_{K_a}:K_a\to K_a\) is a continuous self-map of a nonempty compact convex set. Brouwer's fixed-point theorem yields \(z\in K_a\) such that
\begin{equation}
\label{eq:return-fixed-point}
\Psi^{2T}(z)=z.
\end{equation}
The strict inclusion above implies \(z\in\B_\varepsilon(a)\), while \(\Psi^T(z)\in\B_\varepsilon(b)\). Since \(\abs{a-b}>2\delta>2\varepsilon\), these two balls are disjoint, so \(z\) lies on a nonconstant \(2T\)-periodic orbit.

For every \(s\ge0\), the semigroup identity and \eqref{eq:return-fixed-point} give
\begin{equation}
\label{eq:orbitwise-fixed}
\Psi^{2T}(\Psi^s(z))
=\Psi^s(\Psi^{2T}(z))
=\Psi^s(z).
\end{equation}
Let \(\gamma(s):=\Psi^s(z)\), \(0\le s\le T\). Its initial point lies in \(\B_\varepsilon(a)\), whereas
\[
|\gamma(T)-a|
\ge |a-b|-|\gamma(T)-b|
>2\delta-\varepsilon>\delta.
\]
By continuity, there exists \(s_*\in(0,T)\) such that \(|\gamma(s_*)-a|=(\varepsilon+\delta)/2\). Set \(q_{\rm orb}:=\gamma(s_*)\). Then \(q_{\rm orb}\in C_a\setminus\overline\B_\varepsilon(a)\). Equation \eqref{eq:return-compression} gives \(\Psi^{2T}(q_{\rm orb})\in\B_\varepsilon(a)\), while \eqref{eq:orbitwise-fixed} gives \(\Psi^{2T}(q_{\rm orb})=q_{\rm orb}\), a contradiction.
\end{proof}

Figure~\ref{fig:point-region} summarizes the closed-ball return construction and the periodic-orbit contradiction.

\begin{figure}[!ht]
\centering
\includegraphics[width=.8\textwidth]{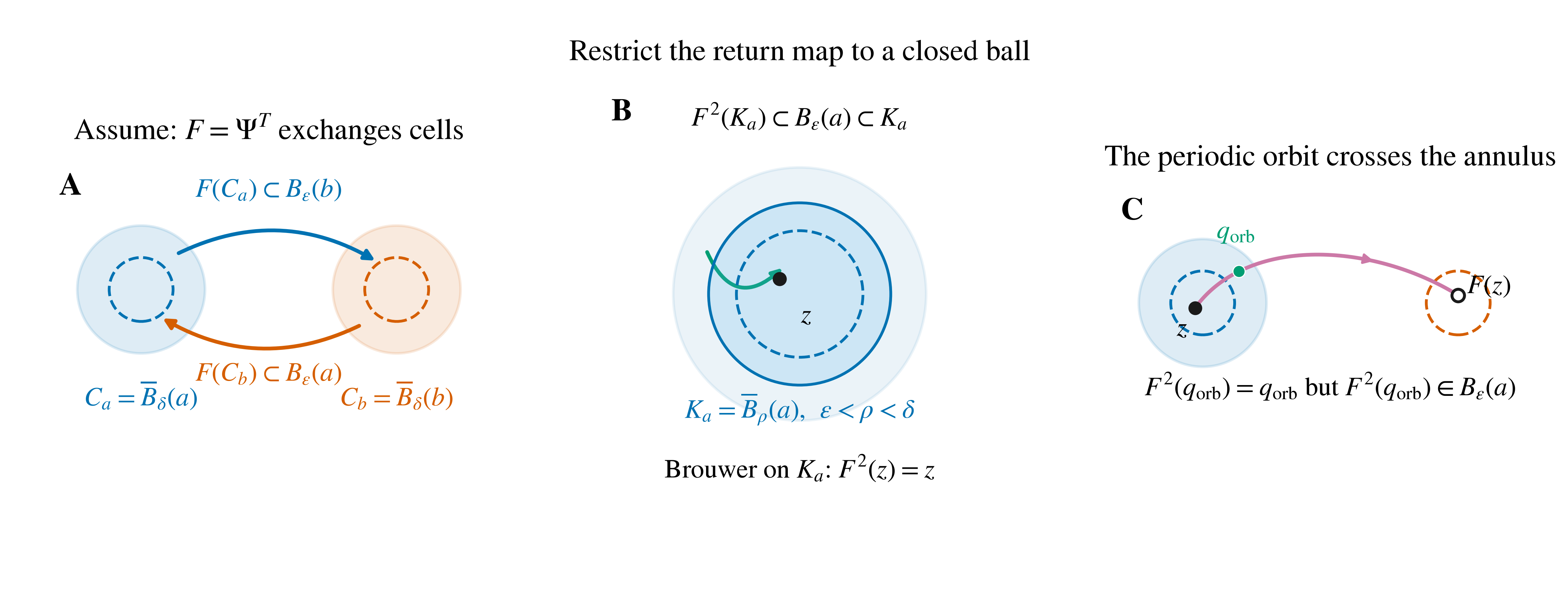}
\caption{Why an autonomous semiflow cannot exchange and compress two balls. Write $F=\Psi^T$. (A) Assume $F$ sends each closed source ball of radius $\delta$ into the open ball of radius $\varepsilon<\delta$ around the other center. (B) For $\varepsilon<\rho<\delta$, the second iterate maps the closed ball $K_a$ of radius $\rho$ into itself, so Brouwer's theorem gives $F^2(z)=z$. (C) Since $F^2$ commutes with $\Psi^s$, every point on the semiflow path from $z$ to $F(z)$ is fixed by $F^2$. This path crosses the annulus at $q$ with $\varepsilon<|q-a|<\delta$, but the assumed return inclusion requires $F^2(q)$ to lie inside the $\varepsilon$-ball, a contradiction.}
\label{fig:point-region}
\end{figure}

\begin{proof}[Proof of Corollary~\ref{cor:sharp-threshold}]
The implication at \(\delta=0\) is Corollary~\ref{thm:qualitative-interpolation}.  Let \(\delta>0\) and suppose that \((\mathsf P_\delta)\) holds. Choose \(a,b\in\R^N\) with \(\abs{a-b}>2\delta\) and \(0<\varepsilon<\delta\). Applying \((\mathsf P_\delta)\) with the two target centers exchanged gives
\[
\Phi^T(\overline\B_\delta(a))\subset\B_\varepsilon(b),\qquad
\Phi^T(\overline\B_\delta(b))\subset\B_\varepsilon(a),
\]
contradicting Theorem~\ref{thm:two-cell}. Applying SCC to the compact convex cells \(\overline\B_\delta(a)\) and \(\overline\B_\delta(b)\) would imply, a fortiori, the open-ball routing configuration forbidden by Theorem~\ref{thm:two-cell}. Hence SCC fails as well.
\end{proof}

\begin{remark}[Structural meaning of the two-cell obstruction]
The obstruction is structural rather than a capacity bound: autonomy forces \(\Psi^{2T}=\Psi^T\circ\Psi^T\), producing the strict return and periodic-orbit contradiction above.  It is independent of width and parameter magnitude and concerns a radius fixed before the dynamics are chosen; data-dependent or almost-everywhere routing is not excluded.
\end{remark}

\subsection{Metric obstruction under bounded strength}

A distinct limitation follows from quantitative invertibility: bounded time--strength prevents arbitrary collapse of within-cell spread.  The following deterministic push-forward estimate is not a statistical generalization or learning-rate statement.

\paragraph{Within-cell squared-error floor for piecewise-constant targets.}

Let \(q\in\mathbb N\), and let \(A_1,\ldots,A_q\subset\R^N\) be pairwise disjoint Borel sets.  Put \(\cA:=(A_1,\ldots,A_q)\), let \(\mu\) be a probability measure supported on their union, and put \(\pi_k:=\mu(A_k)>0\) and \(\mu_k:=\pi_k^{-1}\mu|_{A_k}\). Assume that every \(\mu_k\) has finite second moment. For target vectors \(r_1,\ldots,r_q\in\R^N\), define \(f_*(x):=r_k\) for \(x\in A_k\), and define \(f_*\) arbitrarily outside \(\bigcup_{k=1}^qA_k\). Given an affine map \(S(x)=B(x-x_*)\), the normalized squared population risk of a predictor \(F:\R^N\to\R^N\) is
\begin{equation}
\label{eq:population-risk}
\mathcal R_S(F):=
\int_{\R^N}\abs{B(F(x)-f_*(x))}^2\dd\mu(x).
\end{equation}
Let \(\nu_k:=S_\#\mu_k\).  The corresponding normalized within-cell spread is
\begin{equation}
\label{eq:within-cell-spread}
\mathcal W_S(\mu,\cA):=
\sum_{k=1}^q\pi_k\Var(\nu_k),
\qquad
\overline z_\nu:=\int z\,\dd\nu(z),
\qquad
\Var(\nu):=\int\abs{z-\overline z_\nu}^2\dd\nu(z).
\end{equation}
The identity
\begin{equation}
\label{eq:variance-pair}
\Var(\nu)
=\frac12\iint\abs{z-z'}^2\dd\nu(z)\dd\nu(z')
\end{equation}
will be used below.

Even outside the two-cell exchange configuration, quantitative invertibility imposes the following compression cost.

\begin{proposition}[Bounded-distortion obstruction to neighborhood compression]
\label{prop:risk-floor}
Let \(V_\theta\in\NN_N\), let \(F_\theta:=\Phi_{V_\theta}^T\), and let \(S(x)=B(x-x_*)\), with \(B\in\mathrm{GL}(N)\), be an affine change of coordinates.  For the piecewise-constant regression problem in \eqref{eq:population-risk}--\eqref{eq:within-cell-spread},
the conjugated terminal map $\widetilde F=S\circ F_\theta\circ S^{-1}$ obeys the co-Lipschitz estimate
\begin{equation}
|\widetilde F(z)-\widetilde F(z')|\ge e^{-T\Str_S(\theta)}|z-z'|.
\end{equation}
Consequently, bounded time--strength prevents arbitrary contraction of pairwise distances and within-region variance. As an application,
\begin{equation}
\label{eq:risk-floor}
\mathcal R_S(F_\theta)
\ge
e^{-2T\Str_S(\theta)}
\mathcal W_S(\mu,\cA).
\end{equation}
Consequently, for every budget \(Q\ge0\),
\begin{equation}
\label{eq:risk-budget}
\inf_{\theta:\,T\Str_S(\theta)\le Q}
\mathcal R_S(F_\theta)
\ge
e^{-2Q}\mathcal W_S(\mu,\cA).
\end{equation}
Equivalently, every parameter list with \(\mathcal R_S(F_\theta)>0\) satisfies the compression-cost bound
\begin{equation}
\label{eq:compression-cost}
T\Str_S(\theta)
\ge
\frac12
\log_+\!\left(
\frac{\mathcal W_S(\mu,\cA)}
{\mathcal R_S(F_\theta)}
\right).
\end{equation}
\end{proposition}

\begin{corollary}[Within-cell squared-error floor for piecewise-constant targets]
\label{cor:exact-risk}
Let \(N\ge2\) and \(T>0\).  Assume, in addition, that the targets \(r_1,\ldots,r_q\) are pairwise distinct, and choose one point \(x_k\in A_k\) from each cell.  Apply Theorem~\ref{thm:quantitative-exactification} to \(\cD=\{(x_k,r_k)\}_{k=1}^q\), and let \(S\) be the affine normalization supplied by that theorem.  The resulting exact interpolant satisfies
\begin{equation}
\mathcal R_S(F_\theta)
\ge
e^{-2\mathfrak U_N(\cD)}
\mathcal W_S(\mu,\cA).
\end{equation}
\end{corollary}

\begin{proof}[Proof of Proposition~\ref{prop:risk-floor}]
Conjugate the field by \(S\), setting \(\widetilde V(z):=BV_\theta(B^{-1}z+x_*)\) and \(\widetilde F:=\Phi_{\widetilde V}^T =S\circ F_\theta\circ S^{-1}\).  By \eqref{eq:functional-strength}, \(\Lip(\widetilde V)\le\Str_S(\theta)\).  Because \(\widetilde V\) is globally Lipschitz, its flow is defined for both positive and negative times. Gr\"onwall's inequality for the forward flow gives the usual upper Lipschitz estimate.  Applying the same estimate to the inverse flow \(\Phi_{\widetilde V}^{-T}\) and rearranging gives
\begin{equation}
\label{eq:normalized-co-lipschitz}
\abs{\widetilde F(z)-\widetilde F(z')}
\ge
e^{-T\Str_S(\theta)}\abs{z-z'}.
\end{equation}

Fix a cell \(A_k\), and let \(Z,Z'\) be independent random variables with law \(\nu_k=S_\#\mu_k\), and set \(Z_T:=\widetilde F(Z)\), \(Z_T':=\widetilde F(Z')\). The variance decomposition and \eqref{eq:normalized-co-lipschitz} imply
\[
\begin{aligned}
\int
\abs{\widetilde F(z)-S(r_k)}^2
\dd\nu_k(z)
&\ge
\Var(Z_T)
=
\frac12\mathbb E\abs{Z_T-Z_T'}^2 \ge
\frac12e^{-2T\Str_S(\theta)}
\mathbb E\abs{Z-Z'}^2 =
e^{-2T\Str_S(\theta)}
\Var(\nu_k),
\end{aligned}
\]
where the last identity is \eqref{eq:variance-pair}.  Multiply by \(\pi_k\) and sum over \(k\) to obtain \eqref{eq:risk-floor}.  The budget estimate \eqref{eq:risk-budget} is immediate.  Rearranging \eqref{eq:risk-floor}, and using the trivial lower bound \(T\Str_S(\theta)\ge0\), gives \eqref{eq:compression-cost}.
\end{proof}

\begin{proof}[Proof of Corollary~\ref{cor:exact-risk}]
The selected dataset is admissible because the cells are disjoint and the targets are pairwise distinct.  Theorem~\ref{thm:quantitative-exactification} provides an exact interpolant satisfying \(T\Str_S(\theta)\le\mathfrak U_N(\cD)\). Insert this upper bound into \eqref{eq:risk-floor}.
\end{proof}

\begin{remark}[Two-level obstruction and relative compression cost]
Theorem~\ref{thm:two-cell} is strength-independent and structural, whereas Proposition~\ref{prop:risk-floor} is metric and applies beyond the two-cell pattern. If an exact interpolant satisfies \(T\Str_S(\theta_Q)\le Q\), then
\begin{equation}
\mathcal R_S(F_{\theta_Q})\ge e^{-2Q}\mathcal W_S(\mu,\cA).
\end{equation}
Hence, when \(\mathcal W_S(\mu,\cA)>0\), the requirement \(\mathcal R_S(F_{\theta_Q})\le\varepsilon_{\rm rel}^2\mathcal W_S(\mu,\cA)\) forces
\begin{equation}
T\Str_S(\theta_Q)\ge\log(1/\varepsilon_{\rm rel}).
\end{equation}
\end{remark}

\FloatBarrier

\section{Discussion, scope, and open problems}
\label{sec:learning-consequences}

The results above concern finite-sample exact interpolation and its autonomous
neighborhood-scale limitations.  They do not establish statistical
generalization, density estimation, or distributional approximation.  Repeated
labels require a noninjective readout or target regions, and
Proposition~\ref{prop:risk-floor} is a deterministic bounded-distortion estimate.
This section records the principal distinctions and open directions.

\subsection{Comparisons and scope}

\emph{Measure transport versus labeled routing.} Autonomous transport of one
probability measure to another is a different problem from routing prescribed
labels attached to finitely many particles.  A pushforward specifies a terminal
distribution but does not prescribe a terminal point for every initial particle.
Accordingly, the two-cell obstruction does not transfer verbatim to measure
transport.  Relevant exact and approximate transport results are
\cite{denittiFernandezReal2025,lee2026beckmann,alvareztransport2025}; their
relation to autonomous neural or vector-field classes under common-time and
regularity constraints remains a separate controllability question.

\emph{Approximation inputs versus endpoint exactness.} Classical density results
for nonpolynomial ridge networks provide accurate vector-field approximants
\cite{cybenko1989,hornik1991,pinkus1999}; the quantitative ReLU argument used
here additionally draws on coefficient-controlled approximation and
smoothness-to-variation estimates \cite{klusowski2018,siegel2025,yangzhou2025}.
Exact satisfaction of the endpoint system, however, comes from the
finite-dimensional correction and degree argument, not from density alone.
Preserving finitely many equilibria during vector-field approximation
\cite{pacifico2026} is likewise distinct from prescribing arbitrary
nonstationary values of one common autonomous time map.

\subsection{Open problems and research directions}

\emph{Autonomous time maps and weaker neighborhood notions.} Corollary~\ref{cor:uniform-uap-obstruction} rules out uniform universal approximation of the full class of continuous maps, but it leaves open a characterization of maps in the uniform closure of autonomous time maps. The two-ball obstruction gives one condition stable under small uniform errors. A useful next step is to characterize this closure within specified classes of smooth invertible targets and to compare it with closure in weaker topologies. At the neighborhood scale, it remains to determine which data-dependent radii, target regions, or readouts permit useful alternatives to universal fixed-radius routing.

\emph{Model classes, closure, and stability.} Theorem~\ref{thm:abstract-exactification} uses linear closure so that the base approximant and finitely many endpoint corrections remain in the class.  It is natural to ask whether weaker algebraic or local closure conditions suffice, and whether analogous exactification principles hold for fixed-width, sparse, or otherwise nonlinear parameterized dictionaries.  Separately, the theorem produces an interpolating field for each admissible dataset but does not select one continuously: understanding continuous or Lipschitz selection, and the deterioration of endpoint-frame conditioning near collisions or loss of separation, is open.

\emph{Readouts and measure-level controllability.} Repeated class labels cannot be represented directly by an injective flow map, but a noninjective readout, target regions, or an augmented state may remove this immediate obstruction \cite{ruizbaletAffili2022}.  Which interpolation, margin, or neighborhood guarantees survive after such a readout?  At measure level, for which pairs $(\mu_0,\mu_1)$ can a prescribed autonomous neural or vector-field class satisfy $F_\#\mu_0=\mu_1$ at one common terminal time, exactly or approximately, and which injectivity, regularity, support-geometry, and common-time restrictions remain after labeled endpoint constraints are removed?  The answer may differ for atomic, absolutely continuous, and singular measures.

\emph{Complexity and construction.} The quantitative envelope is deliberately coarse.  Improving its dependence on sample number and affine geometry, proving matching lower bounds for width and path strength, and finding algorithmically implementable and numerically stable exact-interpolation procedures remain open; the endpoint coordinates developed here may provide useful variables for such constructions.

\subsection{Conclusion}

In dimension at least two, uniform approximation of smooth compactly supported vector fields within a linear class whose members generate global flows is sufficient for exact finite interpolation. The proof combines a smooth reference construction with localized endpoint corrections and a degree argument. It applies in particular to shallow ReLU fields and gives quantitative bounds on width and normalized coefficient strength. The two-ball obstruction shows that this interpolation property does not extend to universal control of fixed neighborhoods. Determining the optimal interpolation cost and the maps that remain approximable under autonomy are natural next questions.

\subsection*{Acknowledgements}
The authors thank Antonio \'Alvarez-L\'opez for fruitful suggestions. Enrique Zuazua was supported by the European Research Council (ERC) under the European Union's Horizon Europe research and innovation programme (grant agreement No.~101096251, CoDeFeL); the Air Force Office of Scientific Research under award number FA8655-22-1-7012; the Alexander von Humboldt Professorship program; the European Union's Horizon Europe MSCA project ModConFlex (HORIZON-MSCA-2021-DN-01, project 101073558); SURE-AI: The Norwegian Centre for Sustainable, Risk-Averse, and Ethical AI, grant 357482, Research Council of Norway; Grant PID2023-146872OB-I00-DyCMaMod of MICIU (Spain); and the COST Actions CA24122, \emph{Multiscale Stochastics, Patterns, and Analysis of Combinatorial Environments}, and CA24136, \emph{Interactions between Control Theory and Machine Learning}.

\FloatBarrier

\appendix
\setcounter{section}{0}
\section{Geometric construction and endpoint-correction estimates}
\label{app:finite-transitivity}

This appendix records the geometric and fixed-order estimates for the reference realization and endpoint correction directions constructed in Section~\ref{sec:exactification}.

\textbf{Constant and norm conventions.}  An unsubscripted \(C\) denotes a constant depending only on \(N\) and, where stated, on a fixed derivative order; its value may change between displayed lines. For a derivative of order \(q\), \(\norm{D^qH}_\infty\) denotes the supremum of its \(q\)-linear operator norm (and \(\norm{DH}_\infty\) the usual operator norm). Fixed dimension-dependent equivalence constants between this convention and coordinatewise \(C^s\) norms are absorbed into named bounds. Named constants are retained when they enter the quantitative ledger in Appendix~\ref{app:quantitative-closure}.

\subsection{Localized cylinder pushes and derivative bounds}

\begin{lemma}[Quantitative axial push]
\label{lem:quantitative-axial-push}
For every \(\ell>0\), \(0<r\le1\), and fixed integer \(s\ge1\), there exists \(h_{\ell,r}\in\Diff^\infty(\R)\) such that
\[
h_{\ell,r}(0)=\ell,
\qquad
h_{\ell,r}(\sigma)=\sigma
\quad\text{for }\sigma\notin[-r,\ell+r],
\]
and, with constants depending only on \(s\),
\[
\inf h_{\ell,r}'\ge c_s\frac{r}{\ell+r},
\qquad
1+\max_{1\le q\le s}
\left(
\|h_{\ell,r}^{(q)}\|_\infty+
\|(h_{\ell,r}^{-1})^{(q)}\|_\infty
\right)
\le
\left(\frac{C_s(1+\ell)}{r}\right)^{C_s}.
\]
\end{lemma}

\begin{proof}
Fix a smooth step function \(\chi\in C^\infty(\R;[0,1])\) that vanishes on \((-\infty,0]\), equals one on \([1,\infty)\), and is flat at \(0\) and \(1\). Put
\[
\delta:=\frac{r^2}{100(\ell+r)}
\]
and define
\[
\psi_-(\sigma)
:=\chi\!\left(\frac{\sigma+r}{\delta}\right)
\chi\!\left(\frac{-\sigma}{\delta}\right),
\qquad
\psi_+(\sigma)
:=\chi\!\left(\frac{\sigma}{\delta}\right)
\chi\!\left(\frac{\ell+r-\sigma}{\delta}\right).
\]
Then \(\psi_-\psi_+\equiv0\),
\(\supp\psi_-\subset[-r,0]\),
\(\supp\psi_+\subset[0,\ell+r]\), and, writing
\(I_\pm:=\int_\R\psi_\pm\),
\[
r-2\delta\le I_-\le r,
\qquad
\ell+r-2\delta\le I_+\le\ell+r.
\]
Set
\[
\lambda(\sigma)
:=1+\frac{\ell}{I_-}\psi_-(\sigma)
 -\frac{\ell}{I_+}\psi_+(\sigma),
\qquad
h_{\ell,r}(\sigma)
:=-r+\int_{-r}^{\sigma}\lambda(\tau)\,\dd\tau
\]
on \([-r,\ell+r]\), and extend by the identity outside this interval. The normalized masses give
\(h_{\ell,r}(0)=\ell\) and \(h_{\ell,r}(\ell+r)=\ell+r\). Moreover, \(\lambda\ge1\) on \((-r,0)\), while on \((0,\ell+r)\), since \(\delta\le r/100\),
\[
\lambda
\ge1-\frac{\ell}{I_+}
\ge\frac{r-2\delta}{\ell+r-2\delta}
\ge\frac{49}{50}\frac{r}{\ell+r}.
\]
Thus \(h_{\ell,r}\) is an increasing smooth diffeomorphism.

Let \(A:=(1+\ell)/r\ge1\). Since \(\delta^{-1}\le100A^2\), fixed-cutoff scaling gives \(\|\psi_\pm^{(q)}\|_\infty\le C_qA^{2q}\). Also \(\ell/I_-\le2A\) and \(\ell/I_+\le2\). Hence the derivatives of \(h_{\ell,r}\) through every fixed order are bounded by a fixed power of \(A\). The lower bound on \(h_{\ell,r}'\) gives \(\|(h_{\ell,r}^{-1})'\|_\infty\le C A\); repeated differentiation of \(h_{\ell,r}\circ h_{\ell,r}^{-1}=\Id\) and the one-dimensional Fa\`a di Bruno formula give the same fixed-power bound for higher inverse derivatives. Enlarging \(C_s\) proves the claim.
\end{proof}

\begin{lemma}[Finite-order composition and pushforward bounds]
	\label{lem:finite-order-composition}
	Fix $N\ge1$ and an integer $s\ge1$.  There exists
	$C_{\mathrm{comp},N,s}\ge1$, depending only on $N$ and $s$, with the
	following properties.
	
	\begin{enumerate}[label=\textnormal{(\roman*)}]
		\item If $P_1,\ldots,P_n$ are $C^s$ diffeomorphisms and, for some
		$D\ge1$,
		\[
		1+\max_{1\le j\le n}\max_{1\le q\le s}
		\left(
		\norm{D^qP_j}_\infty+\norm{D^qP_j^{-1}}_\infty
		\right)\le D,
		\]
		then, for $H=P_n\circ\cdots\circ P_1$,
		\[
		1+\max_{1\le q\le s}
		\left(
		\norm{D^qH}_\infty+\norm{D^qH^{-1}}_\infty
		\right)
		\le
		\bigl(C_{\mathrm{comp},N,s}D\bigr)^{C_{\mathrm{comp},N,s}n}.
		\]
		
		\item If $H$ is a $C^{s+1}$ diffeomorphism,
		$u\in C_c^s(\R^N;\R^N)$, and
		\[
		\mathfrak H_{s+1}
		:=
		1+\max_{1\le q\le s+1}
		\left(
		\norm{D^qH}_\infty+\norm{D^qH^{-1}}_\infty
		\right),
		\]
		then
		\[
		\norm{H_*u}_{C^s}
		\le
		C_{\mathrm{comp},N,s}
		\bigl(1+\norm u_{C^s}\bigr)
		\mathfrak H_{s+1}^{C_{\mathrm{comp},N,s}}.
		\]
	\end{enumerate}
\end{lemma}

\begin{proof}
	For part~\textnormal{(i)}, let $\Pi_q$ denote the set of partitions of $\{1,\ldots,q\}$.  The set-partition form of the multivariate Fa\`a di Bruno formula gives
	\[
	\begin{aligned}
		D^q(F\circ G)(x)[v_1,\ldots,v_q]
		=
		\sum_{\pi\in\Pi_q}
		D^{|\pi|}F(G(x))
		\Bigl[
		D^{|B|}G(x)[v_i]_{i\in B}
		\Bigr]_{B\in\pi}.
	\end{aligned}
	\]
	The point needed below is that every term satisfies
	\(
	\sum_{B\in\pi}|B|=q.
	\)
	
	Set
	\[
	H_k:=P_k\circ\cdots\circ P_1,
	\qquad
	k=1,\ldots,n.
	\]
	We claim that, for a sufficiently large constant
	$C_0=C_0(N,s)\ge2$,
	\[
	\norm{D^qH_k}_\infty
	\le
	(C_0D)^{C_0qk},
	\qquad
	1\le q\le s,\quad 1\le k\le n.
	\]
	The case $k=1$ follows from the hypothesis.  Suppose the claim holds for $H_k$.  Since $H_{k+1}=P_{k+1}\circ H_k$, Fa\`a di Bruno yields
	\[
	\begin{aligned}
		\norm{D^qH_{k+1}}_\infty
		&\le
		C_{N,s}
		\sum_{\pi\in\Pi_q}
		\norm{D^{|\pi|}P_{k+1}}_\infty
		\prod_{B\in\pi}
		\norm{D^{|B|}H_k}_\infty\\
		&\le
		C_{N,s}|\Pi_q|D
		\max_{\pi\in\Pi_q}
		\prod_{B\in\pi}
		(C_0D)^{C_0|B|k}\\
		&=
		C_{N,s}|\Pi_q|D
		(C_0D)^{C_0qk}.
	\end{aligned}
	\]
	Here the last identity is precisely where $\sum_{B\in\pi}|B|=q$ is used.  Since $q\le s$, the numbers $|\Pi_q|$ are bounded in terms of $s$ alone; after enlarging $C_0$,
	\[
	\norm{D^qH_{k+1}}_\infty
	\le
	(C_0D)^{C_0q(k+1)}.
	\]
	Thus the exponent grows only linearly with the number of composed factors, rather than through an iterated $C^s$ composition bound.
	
	Taking $k=n$ and enlarging the constant gives
	\[
	1+\max_{1\le q\le s}\norm{D^qH}_\infty
	\le
	\bigl(C_{\mathrm{comp},N,s}D\bigr)^{
		C_{\mathrm{comp},N,s}n}.
	\]
	Moreover,
	\(
	H^{-1}=P_1^{-1}\circ\cdots\circ P_n^{-1},
	\)
	and the inverse factors satisfy the same derivative hypothesis.
	Applying the same argument to this reverse-order composition gives
	the corresponding estimate for $H^{-1}$ and proves
	part~\textnormal{(i)}.
	
	For part~\textnormal{(ii)}, put
	\[
	K:=H^{-1},
	\qquad
	\Lambda:=\mathfrak H_{s+1},
	\qquad
	A:=DH\circ K,
	\qquad
	B:=u\circ K.
	\]
	Then
	\[
	H_*u=AB.
	\]
	For $0\le r\le s$, Fa\`a di Bruno and $\Lambda\ge1$ give
	\[
	\norm{D^rB}_\infty
	\le
	C_{N,s}\norm u_{C^s}\Lambda^r.
	\]
	Indeed, for $r\ge1$,
	\[
	D^rB(x)[v_1,\ldots,v_r]
	=
	\sum_{\pi\in\Pi_r}
	D^{|\pi|}u(K(x))
	\Bigl[
	D^{|B'|}K(x)[v_i]_{i\in B'}
	\Bigr]_{B'\in\pi},
	\]
	and every term contains at most $r$ factors involving derivatives of
	$K$.
	
	Similarly,
	\[
	\norm{D^rA}_\infty
	\le
	C_{N,s}\Lambda^{r+1},
	\qquad
	0\le r\le s.
	\]
	Indeed, the corresponding Fa\`a di Bruno terms contain $D^{|\pi|+1}H$, together with $|\pi|$ derivatives of $K$. In particular, derivatives of $H$ are needed only through order $r+1\le s+1$, while derivatives of $H^{-1}$ are needed only through order $r\le s$.
	
	The finite-order Leibniz rule now yields, for $0\le q\le s$,
	\[
	\begin{aligned}
		\norm{D^q(H_*u)}_\infty
		&\le
		C_{N,s}
		\sum_{r=0}^q
		\binom qr
		\norm{D^rA}_\infty
		\norm{D^{q-r}B}_\infty\\
		&\le
		C_{N,s}\norm u_{C^s}
		\sum_{r=0}^q
		\binom qr
		\Lambda^{r+1}\Lambda^{q-r}\\
		&\le
		C_{N,s}\norm u_{C^s}\Lambda^{q+1}.
	\end{aligned}
	\]
	The useful bookkeeping here is
	\(
	(r+1)+(q-r)=q+1:
	\)
	the total power of $\Lambda$ is independent of how the $q$ derivatives are distributed between the two factors.  Hence
	\[
	\norm{H_*u}_{C^s}
	\le
	C_{N,s}\norm u_{C^s}\Lambda^{s+1}.
	\]
	Since $\Lambda\ge1$, enlarging $C_{\mathrm{comp},N,s}$ absorbs this stronger estimate into the form stated in the lemma.
	
	Finally,
	\(
	\supp(H_*u)\subset H(\supp u),
	\)
	so $H_*u$ is compactly supported.  This completes the proof.
\end{proof}

\begin{proof}[Proof of Lemma~\ref{lem:finite-transitivity}]
We first treat the case \(K\ge2\).

Choose parking points \(p_j=(R+2+3j)e_1\), \(j=1,\ldots,K\). First move the \(z_j\)'s to the \(p_j\)'s one at a time, and then move the \(p_j\)'s to the \(w_j\)'s. The points that are not currently moving are regarded as obstacles. During the first phase they are a mixture of original and parking points, and during the second phase they are a mixture of parking and target points. The parking points are mutually separated by three and lie at distance at least five from \(\overline\B_R\). Consequently, in every move, both the initial and terminal positions of the moving point are at distance at least \(d_{\rm cfg}\) from every obstacle, and the obstacles are mutually \(d_{\rm cfg}\)-separated.

\emph{A clear two-segment route.} Consider one move from \(a\) to \(b\), with obstacles \(o_1,\ldots,o_{K-1}\). All current points lie in \(\overline\B_{R+3K+2}\). Put \(R_{\rm sh}:=R+4K+1\) and \(\mathcal A_{\rm sh}:=\{\xi:R_{\rm sh}-1<|\xi|<R_{\rm sh}\}\). Writing \(\omega_N\) for the unit-ball volume, set \(c_N^{\rm sh}:=N\omega_N/2^{N-1}\).  Since \(R_{\rm sh}\ge2\),
\[
|\mathcal A_{\rm sh}|
=
\omega_N\bigl(R_{\rm sh}^N-(R_{\rm sh}-1)^N\bigr)
\ge
N\omega_N(R_{\rm sh}-1)^{N-1}
\ge
c_N^{\rm sh}R_{\rm sh}^{N-1}.
\]

Fix a dimension-only constant \(C_N^{\rm cap}\ge1\), large enough to absorb the spherical-cap and radial-integration constants below, including the fixed factor \(8^{N-1}\). Define
\begin{equation}
C_N^{\rm mid}
:=
2+\frac{2C_N^{\rm cap}}{c_N^{\rm sh}},
\qquad
\delta_{\rm tr}
:=
\frac{d_{\rm cfg}}
{C_N^{\rm mid}(1+K)(1+R_{\rm sh})}.
\end{equation}
In particular, \(4\delta_{\rm tr}\le d_{\rm cfg}/2\).

Fix \(a_{\rm end}\in\{a,b\}\) and one obstacle \(o_j\). If \([a_{\rm end},\xi]\), with \(\xi\in\mathcal A_{\rm sh}\), intersects \(\B_{4\delta_{\rm tr}}(o_j)\), choose an intersection point \(q\). Since \(4\delta_{\rm tr}<|o_j-a_{\rm end}|\), the intersection does not occur at \(a_{\rm end}\), and
\[
(q-a_{\rm end})^\top(o_j-a_{\rm end})
\ge |o_j-a_{\rm end}|\bigl(|o_j-a_{\rm end}|-4\delta_{\rm tr}\bigr)>0.
\]
Thus the angle \(\theta_{\rm ang}\) between \(\xi-a_{\rm end}\) and \(o_j-a_{\rm end}\) satisfies \(0\le\theta_{\rm ang}<\pi/2\) and \(\sin\theta_{\rm ang}\le4\delta_{\rm tr}/|o_j-a_{\rm end}| \le4\delta_{\rm tr}/d_{\rm cfg}\le1/2\).  Hence \(\theta_{\rm ang}\le2\sin\theta_{\rm ang}\le8\delta_{\rm tr}/d_{\rm cfg}\), so the direction \((\xi-a_{\rm end})/|\xi-a_{\rm end}|\) belongs to a spherical cap of angular radius
\[
\theta_0:=\frac{8\delta_{\rm tr}}{d_{\rm cfg}}\le1.
\]
For any \(e\in\mathbb S^{N-1}\), spherical coordinates give
\[
\mathcal H^{N-1}\!\left(
\{\omega\in\mathbb S^{N-1}:\angle(\omega,e)\le\theta_0\}
\right)
=|\mathbb S^{N-2}|\int_0^{\theta_0}(\sin\theta)^{N-2}\,\dd\theta
\le\frac{|\mathbb S^{N-2}|}{N-1}\theta_0^{N-1}.
\]
Moreover, \(|a_{\rm end}|\le R+3K+2<R_{\rm sh}\), so every
\(\xi\in\mathcal A_{\rm sh}\) satisfies
\(|\xi-a_{\rm end}|\le2R_{\rm sh}\). Enlarging the bad set to the
corresponding cone and integrating radially about \(a_{\rm end}\) therefore
gives
\[
\begin{aligned}
|\{\xi\in\mathcal A_{\rm sh}:[a_{\rm end},\xi]
\cap\B_{4\delta_{\rm tr}}(o_j)\ne\varnothing\}|
&\le
\frac{|\mathbb S^{N-2}|}{N-1}\theta_0^{N-1}
\int_0^{2R_{\rm sh}}r^{N-1}\,\dd r\\
&\le
C_N^{\rm cap}R_{\rm sh}^N
\left(\frac{\delta_{\rm tr}}{d_{\rm cfg}}\right)^{N-1}.
\end{aligned}
\]
After a union bound over the two endpoints and \(K-1\) obstacles, the ratio of bad volume to \(|\mathcal A_{\rm sh}|\) is, since \(0<\delta_{\rm tr}/d_{\rm cfg} =[C_N^{\rm mid}(1+K)(1+R_{\rm sh})]^{-1}<1\) and \(N\ge2\), at most
\[
\frac{2C_N^{\rm cap}}{c_N^{\rm sh}}KR_{\rm sh}
\left(\frac{\delta_{\rm tr}}{d_{\rm cfg}}\right)^{N-1}
\le
\frac{2C_N^{\rm cap}}
{c_N^{\rm sh}C_N^{\rm mid}}
\frac{K}{1+K}
\frac{R_{\rm sh}}{1+R_{\rm sh}}
<1.
\]
Thus some \(\xi_{\rm sh}\in\mathcal A_{\rm sh}\) makes both \([a,\xi_{\rm sh}]\) and \([\xi_{\rm sh},b]\) stay at distance at least \(4\delta_{\rm tr}\) from every obstacle. This angular argument is the only place where \(N\ge2\) is essential.

\emph{One straight-cylinder push.} We now use the quantitative axial construction above.

Now consider an oriented segment of length \(\ell_{\rm seg}>0\) and a radius \(0<r_{\rm cyl}\le1\). After an orthogonal change of coordinates, write points as \((\sigma,\xi)\in\R\times\R^{N-1}\), with endpoints \((0,0)\) and \((\ell_{\rm seg},0)\), and take
\(h_{\rm ax}:=h_{\ell_{\rm seg},r_{\rm cyl}}\) from Lemma~\ref{lem:quantitative-axial-push}.

Choose a fixed \(\zeta\in C_c^\infty(\B_1^{N-1})\), \(0\le\zeta\le1\), with \(\zeta\equiv1\) on \(\B_{1/3}^{N-1}\), and put \(\zeta_r(\xi):=\zeta(\xi/r_{\rm cyl})\).  Define
\[
\mathscr P_{\rm cyl}(\sigma,\xi)
:=
\bigl(\sigma+\zeta_r(\xi)(h_{\rm ax}(\sigma)-\sigma),\xi\bigr).
\]
Its axial derivative is \((1-\zeta_r)+\zeta_r h_{\rm ax}'>0\), so \(\mathscr P_{\rm cyl}\) is a diffeomorphism, is the identity near the cylinder boundary, and sends \((0,0)\) to \((\ell_{\rm seg},0)\). For clarity, the inverse axial coordinate \(\sigma=\sigma(\tau,\xi)\) is determined by \(\tau=\sigma+\zeta_r(\xi)(h_{\rm ax}(\sigma)-\sigma)\). Its axial derivative \(d=1-\zeta_r(\xi)+\zeta_r(\xi)h_{\rm ax}'(\sigma)\) is a convex combination of \(1\) and \(h_{\rm ax}'\), hence is bounded below by a positive constant times \(r_{\rm cyl}/(\ell_{\rm seg}+r_{\rm cyl})\). Implicit differentiation gives \(\partial_\tau\sigma=1/d\) and \(\partial_{\xi_j}\sigma=-\partial_{\xi_j}\zeta_r(\xi)(h_{\rm ax}(\sigma)-\sigma)/d\). The displacement and cutoff bounds, followed by repeated implicit differentiation, bound every fixed-order inverse derivative by a fixed power of \((1+\ell_{\rm seg})/r_{\rm cyl}\). Thus there is a constant \(C_{\rm seg,N,s}\ge1\), depending only on \(N,s\), such that
\begin{equation}
1+\max_{1\le q_{\rm der}\le s}
\left(
\norm{D^{q_{\rm der}}\mathscr P_{\rm cyl}}_\infty
+
\norm{D^{q_{\rm der}}\mathscr P_{\rm cyl}^{-1}}_\infty
\right)
\le
\left(
\frac{C_{\rm seg,N,s}(1+\ell_{\rm seg})}{r_{\rm cyl}}
\right)^{C_{\rm seg,N,s}}.
\label{eq:one-segment-bound}
\end{equation}
In the routed construction take \(r_{\rm cyl}=\delta_{\rm tr}/2\).  The support then lies within distance \(<\delta_{\rm tr}\) of the segment and hence fixes every current obstacle.

\emph{Composition for \(K\ge2\).} There are at most \(2K\) point moves and hence at most \(4K\) segment pushes. Every routed segment has \(\ell_{\rm seg}\le2R_{\rm sh}\), while \(r_{\rm cyl}^{-1}=2C_N^{\rm mid}(1+K)(1+R_{\rm sh})/d_{\rm cfg}\) and \(1+R_{\rm sh}\le4(1+K)(1+R)\). Consequently,
\begin{equation}
\frac{1+\ell_{\rm seg}}{r_{\rm cyl}}
\le
\left(
\frac{
C_N^{\rm len}(1+K)(1+R)
}{
d_{\rm cfg}
}
\right)^3,
\qquad
C_N^{\rm len}
:=
8\max\{1,C_N^{\rm mid}\}.
\label{eq:length-radius-bound}
\end{equation}

Apply Lemma~\ref{lem:finite-order-composition}\textnormal{(i)} to the at most $4K$ segment pushes and to the reverse-order inverse composition.  Choose
\[
C_{\mathrm{tr},N,s}
:=12C_{\mathrm{comp},N,s}
\bigl(1+C_{\mathrm{seg},N,s}+C_N^{\rm len}\bigr).
\]
Combining Lemma~\ref{lem:finite-order-composition}\textnormal{(i)} with \eqref{eq:one-segment-bound} and \eqref{eq:length-radius-bound} gives \eqref{eq:H-bound}. Each push fixes all current obstacles and all already parked or placed points. Moreover, every routed segment lies in \(\B_{R_{\rm sh}}\), and every support cylinder lies in \(\B_{R_{\rm sh}+1}=\B_{R+4K+2}\Subset\Omega_{\rm tr}\). This proves the lemma for \(K\ge2\).

It remains to treat \(K=1\). If \(z_1=w_1\), take \(H=\Id\). Otherwise apply the straight-cylinder construction above directly to the single segment \([z_1,w_1]\), with \(\ell_{\rm seg}=|w_1-z_1|\le2R\) and \(r_{\rm cyl}:=1/4\). No routing is needed. The support cylinder lies within distance \(\sqrt2\,r_{\rm cyl}<1\) of \([z_1,w_1]\), and hence in \(\B_{R+1}\Subset\Omega_{\rm tr}\). By \eqref{eq:one-segment-bound},
\[
1+\max_{1\le q_{\rm der}\le s}
\left(
\norm{D^{q_{\rm der}}H}_\infty+
\norm{D^{q_{\rm der}}H^{-1}}_\infty
\right)
\le
\bigl(
8C_{\mathrm{seg},N,s}(1+R)
\bigr)^{C_{\mathrm{seg},N,s}}.
\]
Since here \(d_{\rm cfg}=1\), \(1+K=2\), and the above choice of \(C_{\mathrm{tr},N,s}\) satisfies \(C_{\mathrm{tr},N,s}\ge12C_{\mathrm{seg},N,s}\), this is bounded by the right-hand side of \eqref{eq:H-bound}. Thus \eqref{eq:H-bound} also holds for \(K=1\), completing the proof.
\end{proof}

\subsection{Quantitative output for the transported reference field}

Assume the normalized configuration satisfies
\[
\cV\subset\overline\B_{1/2},
\qquad \sep(\cV)=\rho,
\qquad 0<\rho\le1.
\]
Put \(R_{\rm vert}:=4K+1\), \(d_{\rm tr}:=\min\{1,4/K,\rho\}\), and \(R_{\rm tr}:=8K+6\).  The standard vertices lie in \(\B_{R_{\rm vert}}\) and have separation at least \(4/K\).  Applying Lemma~\ref{lem:finite-transitivity} with \(R=R_{\rm vert}\) and derivative order \(s_N+3\) gives the diffeomorphism appearing in \eqref{eq:H-data-map}, supported in \(\B_{R_{\rm tr}}\), together with
\begin{equation}
\label{eq:H-envelope}
\begin{aligned}
\mathfrak H
&:=1+\max_{1\le j\le s_N+3}
\left(\norm{D^jH}_\infty+\norm{D^jH^{-1}}_\infty\right),\\
\mathfrak H
&\le
\left[C_{\mathrm{tr},N,s_N+3}(1+K)(1+R_{\rm vert})d_{\rm tr}^{-1}\right]^{C_{\mathrm{tr},N,s_N+3}K}.
\end{aligned}
\end{equation}
Since \(d_{\rm tr}^{-1}\le K/\rho\), this envelope is bounded by \((C_NK/\rho)^{C_NK}\) after enlarging a dimension-only constant.

The pushforward identity
\[
f(x)=DH(H^{-1}x)\,f_0(H^{-1}x)
\]
fits Lemma~\ref{lem:finite-order-composition}\textnormal{(ii)}.  Since $s_N$ depends only on $N$ and the envelope $\mathfrak H$ controls more than the required $s_N+1$ derivatives, choose a dimension-only constant $C_{\mathrm{pc},N}\ge C_{\mathrm{comp},N,s_N}+1$ large enough to absorb the fixed product--composition constants below.  Then
\begin{equation}
\label{eq:transported-field-bound}
M_f:=C_{\mathrm{pc},N}(1+\norm{f_0}_{C^{s_N}})\mathfrak H^{C_{\mathrm{pc},N}},
\qquad
\norm f_{C^{s_N}}\le M_f.
\end{equation}
The compact support of \(H-\Id\) also shows that \(H\) maps its support ball onto itself and that all reference trajectories remain in \(\B_{R_{\rm tr}}\).

We next record the regularity estimates for the private tubes and endpoint-frame directions from Section~\ref{sec:endpoint-frame-module}.

\subsection{Private-tube geometry}
\label{app:endpoint-coordinates}

Take \(\varepsilon_{\rm tube}:=1/(100K)\).  For \(|z|<2\varepsilon_{\rm tube}\), one has \(3/4<1+z_1<5/4\), while the angular interval swept by \(I_{\rm tube}\) has length at most \(\pi/2\).  The longitudinal derivative of \(\Theta_i^0\) has size comparable to \(\omega_{\nu(i)}\asymp n_{\nu(i)}^{-1}\), and the transverse derivatives form an orthogonal frame up to uniformly bounded factors.  Hence \(\Theta_i^0\) is injective on \(I_{\rm tube}\times\B_{2\varepsilon_{\rm tube}}^{N-1}\), its derivative has singular values bounded below by \(1/K\), and it is a smooth diffeomorphism onto \(O_i^0\).

For \(t\in\overline I_{\rm tube}\), the centerline is at distance at least \(2\sin(\pi/(4n_{\nu(i)}))\ge1/K\) from every different constrained arc on the same cycle; arcs in distinct components are separated by at least two.  Since the tube radius is less than \(1/(50K)\), this proves the privacy property \eqref{eq:standard-tube-privacy}.  Repeated differentiation of \(\Theta_i^0\) and its inverse yields, for every fixed integer \(s\ge1\),
\begin{equation}
\label{eq:tube-derivative-envelope}
1+\norm{\Theta_i^0}_{C^{s+1}}
+\norm{(\Theta_i^0)^{-1}}_{C^{s+1}(O_i^0)}
\le C_{\mathrm{tube},N,s}K^{2s+2}.
\end{equation}
The power of \(K\) records the only small singular value, namely the longitudinal one.

\subsection{Smooth zero extension and endpoint-correction norms}
\label{app:zero-extension}

For completeness, one may take
\[
\vartheta(t):=
\begin{cases}
\displaystyle
\frac{\exp\!\left(-\frac{1}{1-100(t-1/2)^2}\right)}
{\displaystyle\int_{2/5}^{3/5}\exp\!\left(-\frac{1}{1-100(s-1/2)^2}\right)\dd s},
&|t-1/2|<1/10,\\[3mm]
0,&|t-1/2|\ge1/10,
\end{cases}
\]
and
\[
\zeta_{\rm fr}(z):=
\begin{cases}
\exp\!\bigl(1-1/(1-|z|^2)\bigr),&|z|<1,\\
0,&|z|\ge1,
\end{cases}
\qquad
\zeta_i(z):=\zeta_{\rm fr}(z/\varepsilon_{\rm tube}).
\]
Their supports are compactly contained in the coordinate cylinder, so the fields \(k_{i,a}^0\) defined in Section~\ref{sec:endpoint-frame-module} extend smoothly by zero.  Derivatives of \(\zeta_i\) through order \(s\) cost at most a fixed multiple of \((100K)^s\); derivatives of \(U_i^0(1,t)^{-1}\) are uniformly bounded; and composition with \((\Theta_i^0)^{-1}\) is controlled by \eqref{eq:tube-derivative-envelope}.  Thus
\begin{equation}
\label{eq:standard-frame-Cs}
\max_{i,a}\norm{k_{i,a}^0}_{C^s}
\le C_{\mathrm{fr},N,s}K^{2s^2+4s+2}.
\end{equation}
For a self-loop, the fixed bump construction has the same conclusion with a dimension-only bound.

For the transported directions in \eqref{eq:preconditioned-standard-frame}, the preconditioning matrix $DH(y_i^0)^{-1}$ is bounded by $\mathfrak H$.  Thus the $C^{s_N}$ norm of the standard field inside the pushforward is bounded by a fixed multiple of $\mathfrak H\,C_{\mathrm{fr},N,s_N}K^{2s_N^2+4s_N+2}$.  Applying Lemma~\ref{lem:finite-order-composition}\textnormal{(ii)} and enlarging $C_{\mathrm{pc},N}$ to absorb this additional factor gives
\begin{equation}
\label{eq:transported-frame-bound}
M_h:=C_{\mathrm{pc},N}C_{\mathrm{fr},N,s_N}
K^{2s_N^2+4s_N+2}\mathfrak H^{C_{\mathrm{pc},N}},
\qquad
\max_{1\le\ell\le m}\norm{h_\ell}_{C^{s_N}}\le M_h.
\end{equation}
No additional endpoint argument is needed, since the identity \eqref{eq:ideal-sensitivity-identity} was already established in the body of the paper.

\section{Coefficient-controlled shallow-ReLU approximation}
\label{app:relu-approximation}

This appendix contains the only approximation-theoretic input specific to the shallow ReLU dictionary.  It proves Lemma~\ref{lem:smooth-vector-approximation} with simultaneous control of uniform error, width, and coefficient strength.  No endpoint or degree argument is used here.

\subsection{Coefficient-controlled shallow-ReLU approximation}

Let \(R>0\) and \(u\in C(\overline\B_R)\).  Define the normalized product-dictionary ReLU variation
\begin{equation}
\label{eq:variation-definition}
\gamma_R(u):=\inf\left\{\norm{\lambda}_{\TV}:
u(x)=\int_{\mathbb S^{N-1}\times[-1,1]}
\left(v^\top\frac{x}{R}+\beta_{\rm rid}\right)_+
\dd\lambda(v,\beta_{\rm rid}),\ \forall x\in\overline\B_R\right\},
\end{equation}
where \(\lambda\) ranges over finite signed Radon measures.  If no representation exists, set \(\gamma_R(u)=+\infty\).  We shall also use the joint-sphere dictionary \((a^\top z+b)_+\), \((a,b)\in\mathbb S^N\), employed by the smoothness embedding below.  The product variation in \eqref{eq:variation-definition} is no larger than the corresponding joint-sphere variation, while the latter is at most \(\sqrt2\,\gamma_R(u)\).  Indeed, decompose the joint sphere into \(a\ne0,|b|\le|a|\), \(a\ne0,b>|a|\), \(a\ne0,b<-|a|\), and \(a=0\). On the first set, \((a^\top z+b)_+ =|a|\bigl((a/|a|)^\top z+b/|a|\bigr)_+\). On \(b>|a|\) the atom is affine on \(\overline\B_1\) and, with \(v=a/|a|\),
\begin{equation}
(a^\top z+b)_+
=\frac{|a|+b}{2}(v^\top z+1)_+
+\frac{b-|a|}{2}(-v^\top z+1)_+.
\end{equation}
If \(a=0,b=1\), fix any \(v\in\mathbb S^{N-1}\) and use \(1=\tfrac12(v^\top z+1)_++\tfrac12(-v^\top z+1)_+\), whereas \(a=0,b=-1\) and \(b<-|a|\) give the zero atom.  These transformations do not increase total variation.  Conversely, for \(|v|=1\) and \(|\beta_{\rm rid}|\le1\), \((v^\top z+\beta_{\rm rid})_+ =\sqrt{1+\beta_{\rm rid}^2}\bigl( (v/\sqrt{1+\beta_{\rm rid}^2})^\top z +\beta_{\rm rid}/\sqrt{1+\beta_{\rm rid}^2}\bigr)_+\), whose cost is at most \(\sqrt2\).  This proves the stated comparison.

The next two lemmas record the precise Siegel and Yang--Zhou specializations used below; the scaling is \(z=x/R\), with the preceding joint-sphere/product-dictionary conversion.

\begin{lemma}[Coefficient-controlled uniform approximation]
\label{lem:coefficient-approximation}
There is a finite \(A_N>0\), depending only on \(N\), such that for every \(R>0\), \(u\in C(\overline\B_R)\), strict variation budget \(\mathcal B>\gamma_R(u)\), and integer \(n_{\rm rid}\ge1\), there is an \(n_{\rm rid}\)-term scalar ReLU network
\[
u_{n_{\rm rid}}(x)=\sum_{j=1}^{n_{\rm rid}} c_j
\left(v_j^\top\frac{x}{R}+\beta_j^{\rm rid}\right)_+
\]
satisfying
\begin{align}
&\sum_{j=1}^{n_{\rm rid}}|c_j|\le \mathcal B,
\quad \Lip(u_{n_{\rm rid}})\le \mathcal B/R,\quad
\sum_{j=1}^{n_{\rm rid}}|c_j|
\sqrt{R^{-2}+(\beta_j^{\rm rid})^2}
\le \sqrt{1+R^{-2}}\,\mathcal B.
\label{eq:scalar-strength-control}\\
&\norm{u-u_{n_{\rm rid}}}_{C(\overline\B_R)}
\le A_N\mathcal B n_{\rm rid}^{-1/\kappa_N},
\label{eq:scalar-approx}
\end{align}
Here \(\kappa_N\) is defined in \eqref{eq:relu-approx-exponent}.
\end{lemma}

\begin{proof}[Proof of Lemma~\ref{lem:coefficient-approximation}]
Theorem~3 of Siegel~\cite{siegel2025}, specialized to \(k=1\), derivative order zero, and ambient dimension \(N\), gives the rate \(n_{\rm rid}^{-1/2-3/(2N)}=n_{\rm rid}^{-1/\kappa_N}\) on the unit ball and, in the notation of that result, chooses the approximant in \(\Sigma_{n_{\rm rid}}^1\), so that the sum of the absolute coefficients is at most one.  Apply it to \(u(R\,\cdot)/\mathcal B\), whose product-dictionary variation is strictly smaller than one.  Multiplication by \(\mathcal B\) and the change of variables \(z=x/R\) give \eqref{eq:scalar-approx} and the coefficient bound. Finally, every normalized atom is \(R^{-1}\)-Lipschitz in the physical variable, so \(\Lip(u_{n_{\rm rid}}) \le R^{-1}\sum_{j=1}^{n_{\rm rid}}|c_j|\le \mathcal B/R\). The augmented norm of its inner parameter \((v_j/R,\beta_j^{\rm rid})\) is at most \(\sqrt{R^{-2}+1}\), which proves \eqref{eq:scalar-strength-control}. Zero coefficients pad the representation to exactly \(n_{\rm rid}\) terms, while \(\mathcal B>\gamma_R(u)\) avoids assuming attainment of the variation infimum.  Continuity identifies the source \(L^\infty\) error with the uniform error on the closed ball, so the pullback gives the stated norm on \(\overline\B_R\).
\end{proof}

For an integer \(s\ge1\), define
\begin{equation}
\label{eq:scaled-extension-norm}
\norm{u}_{C_R^s,\mathrm{ext}}
:=\inf\left\{\norm{E}_{C^s(\R^N)}:
E(z)=u(Rz)\text{ for }|z|\le1\right\}.
\end{equation}

\begin{lemma}[Smoothness-to-variation embedding]
\label{lem:smoothness-variation}
If the integer \(s>(N+3)/2\), then there exists \(C_{\mathrm{emb},N,s}<\infty\), depending only on \(N,s\), such that
\begin{equation}
\label{eq:smoothness-variation-bound}
\gamma_R(u)\le C_{\mathrm{emb},N,s}
\norm{u}_{C_R^s,\mathrm{ext}}
\end{equation}
whenever the right-hand side is finite.
\end{lemma}

\begin{proof}[Proof of Lemma~\ref{lem:smoothness-variation}]
Theorem~2.1 of Yang--Zhou~\cite{yangzhou2025}, with \(k=1\), \(d=N\), and smoothness index \(s\), applies because \(s>(N+3)/2\).  In the notation of that theorem it bounds the joint-sphere variation of \(z\mapsto u(Rz)\) by a constant depending only on \(N,s\) times its global-extension H\"older norm.  At the integer index \(s\), their convention writes \(s=(s-1)+1\), so this is a \(C^{s-1,1}\) norm.  It is controlled by the \(C^s\) norm in \eqref{eq:scaled-extension-norm}, by the multivariate mean-value theorem.  The joint-to-product conversion established above does not increase total variation.  Taking the infimum over extensions proves \eqref{eq:smoothness-variation-bound}.  Because both sides are continuous, the almost-everywhere identity in the cited function-space result holds pointwise on \(\overline\B_1\).
\end{proof}

\begin{proof}[Proof of Lemma~\ref{lem:smooth-vector-approximation}]
If \(u=0\), take the empty network.  Suppose \(u\ne0\), and write \(u=(u^{(1)},\ldots,u^{(N)})\).  For each coordinate, the global function \(z\mapsto u^{(k)}(Rz)\) is an admissible extension in \eqref{eq:scaled-extension-norm}, with \(C^s\)-norm at most \((1+R)^s\norm u_{C^s}\). Lemma~\ref{lem:smoothness-variation} therefore gives \(\gamma_R(u^{(k)}) \le C_{\mathrm{emb},N,s}(1+R)^s\norm u_{C^s} =\tfrac12\mathcal B_{R,s}(u)\). Thus \(\mathcal B_{R,s}(u)\) is a strict variation budget for every nonzero coordinate.  The definition of \(n_{\rm app}(R,s;u,\varepsilon_{\rm app})\) gives \(A_N\mathcal B_{R,s}(u) n_{\rm app}(R,s;u,\varepsilon_{\rm app})^{-1/\kappa_N}\le \varepsilon_{\rm app}/\sqrt N\). Apply Lemma~\ref{lem:coefficient-approximation} to each nonzero scalar coordinate with this common budget and term count; use the zero scalar network for zero coordinates.  Concatenating the coordinate networks gives \(V_{\rm app}\).  The Euclidean error is at most \(\sqrt{N(\varepsilon_{\rm app}/\sqrt N)^2}=\varepsilon_{\rm app}\), and the total number of units is at most \(Nn_{\rm app}(R,s;u,\varepsilon_{\rm app})=\mathcal P_{R,s}(u,\varepsilon_{\rm app})\).  Finally, \eqref{eq:scalar-strength-control} gives \(\Str(V_{\rm app}) \le N\sqrt{1+R^{-2}}\,\mathcal B_{R,s}(u) =\mathcal S_{R,s}(u)\), which proves all three assertions.
\end{proof}

\section{Quantitative estimates for the endpoint-correction bridge}
\label{app:stability-degree}

This appendix proves Proposition~\ref{prop:quantitative-interface} by quantifying the approximation region and tolerances in terms of the normalized data geometry and \(\mathfrak B(\mathbf g)\); the degree step is invoked from Section~\ref{sec:exactification}.

For \(\tau,L\ge0\), write
\begin{equation}
\mathfrak g_\tau(L):=\int_0^\tau e^{L(\tau-t)}\,\dd t
=
\begin{cases}
(e^{L\tau}-1)/L,&L>0,\\
\tau,&L=0.
\end{cases}
\end{equation}
Also set \(L_{\rm std}:=\norm{Df_0}_{C^0}\).  All constants involving only \(f_0\), its fixed cutoff template, and the dimension are treated as standard-model constants.

\subsection{Endpoint-correction conditioning}

Let \(f,h_1,\ldots,h_m\) be the reference objects from Propositions~\ref{prop:smooth-reference-module} and~\ref{prop:endpoint-frame-module}. Along the marked standard trajectories, the cutoff is constant and the linearized moving dynamics are rotations; the stationary components have identity fundamental matrices. Thus the standard fundamental matrices have operator norm one. The conjugacy formula \eqref{eq:fundamental-conjugacy} gives
\begin{equation}
M_U:=\max_{1\le i\le M}\int_0^1\norm{U_i(1,t)}_{\op}\,\dd t
\le\norm{DH}_\infty\norm{DH^{-1}}_\infty.
\end{equation}
Choose
\begin{equation}
\label{eq:sensitivity-tolerance}
\varepsilon_*:=
\frac{1}{8\sqrt{mM}\max\{1,M_U\}}.
\end{equation}
Suppose that \(g_1,\ldots,g_m\) satisfy the first condition in \eqref{eq:bridge-hypotheses}.  Since every reference trajectory lies in \(\B_{R_{\rm tr}}\), the definition of \(\mathcal L_f\) gives, column by column,
\[
\norm{\mathcal L_f(g_\ell)-\mathcal L_f(h_\ell)}_{\R^m}
\le \sqrt M\,M_U\varepsilon_*
\le\frac{1}{8\sqrt m}.
\]
Hence, for
\[
\mathsf J_{\rm end}:=[\mathcal L_f(g_1)\ \cdots\ \mathcal L_f(g_m)],
\]
the Frobenius bound and \eqref{eq:ideal-sensitivity-identity} imply
\[
\norm{\mathsf J_{\rm end}-I_m}_{\op}\le\frac18.
\]
Consequently we may fix
\begin{equation}
\sigma_*:=\frac34,
\qquad
|\mathsf J_{\rm end}\alpha|\ge\sigma_*|\alpha|
\quad(\alpha\in\R^m).
\end{equation}

\subsection{Correction-family normalization and quadratic endpoint expansion}

Write
\[
a(z):=(DH(z))^{-1},
\qquad
H_1:=\norm{DH}_\infty,
\qquad
H_{-1}:=\norm a_\infty,
\qquad
H_2:=\norm{D^2H}_\infty.
\]
For \(\widetilde g_\ell(z):=a(z)g_\ell(H(z))\), put \(\widetilde G_\alpha:=\sum_{\ell=1}^m\alpha_\ell\widetilde g_\ell\).  Since \(H-\Id\) is supported in \(\B_{R_{\rm tr}}\), the map \(H\) sends that ball onto itself.  Moreover,
\[
D a(z)[v]
=-a(z)D^2H(z)[v]a(z).
\]
The Lipschitz estimate is global even though the fields \(g_\ell\) need not be bounded on \(\R^N\). Inside \(\B_{R_{\rm tr}}\), the identity \(H(\B_{R_{\rm tr}})=\B_{R_{\rm tr}}\) and linear growth bound \(g_\ell(H(z))\). Outside that ball, \(H\) and \(a\) are the identity and \(Da=0\). Thus the term involving \(Da\) occurs only where the correction field is bounded. Product and chain rules, interpreted almost everywhere for Lipschitz correction fields, give the global Lipschitz estimate.
If \(\mathfrak B(\mathbf g)\le\mathfrak b\), then the linear-growth estimate
\(|g_\ell(x)|\le |g_\ell(0)|+\Lip(g_\ell)|x|\), Cauchy--Schwarz, and the preceding product rule give the aggregate bounds
\begin{equation}
\label{eq:pullback-g-envelope}
\begin{aligned}
A_{\mathfrak b}
&:=(1+R_{\rm tr})H_{-1}\mathfrak b,\\
L_{\mathfrak b}
&:=\Bigl[(1+R_{\rm tr})H_{-1}^2H_2+H_{-1}H_1\Bigr]\mathfrak b,
\end{aligned}
\end{equation}
so that
\[
\left(\sum_{\ell=1}^m|\widetilde g_\ell(0)|^2\right)^{1/2}\le A_{\mathfrak b},
\qquad
\left(\sum_{\ell=1}^m\Lip(\widetilde g_\ell)^2\right)^{1/2}\le L_{\mathfrak b}.
\]
Choose
\begin{equation}
\label{eq:r0}
r_0:=\frac{1}{4(1+A_{\mathfrak b}+L_{\mathfrak b})}.
\end{equation}
For \(|\alpha|\le r_0\), Cauchy--Schwarz gives
\begin{equation}
\label{eq:pullback-smallness}
|\widetilde G_\alpha(0)|\le\frac14,
\qquad
\Lip(\widetilde G_\alpha)\le\frac14.
\end{equation}
The pullback of \(f+G_\alpha\), where \(G_\alpha:=\sum_\ell\alpha_\ell g_\ell\), is exactly \(f_0+\widetilde G_\alpha\).

For \(|\alpha|\le r_0\), let
\[
Z_i^\alpha(t):=H^{-1}(\Phi_{f+G_\alpha}^t(x_i)),
\qquad
\gamma_i^0(t):=\Phi_{f_0}^t(x_i^0).
\]
Thus
\begin{equation}
\label{eq:standard-parameter-ode}
\dot Z_i^\alpha
=f_0(Z_i^\alpha)+\widetilde G_\alpha(Z_i^\alpha),
\qquad
Z_i^\alpha(0)=x_i^0.
\end{equation}
The linear-growth estimate from \eqref{eq:pullback-smallness} and Gronwall's inequality show that all these trajectories remain in \(\overline\B_{R_{\rm par}}\), where
\[
R_{\rm par}:=e^{1/4}R_{\rm vert}
+(\norm{f_0}_{C^0}+1/4)\mathfrak g_1(1/4).
\]
Define
\[
E_{\rm disp}
:=(A_{\mathfrak b}+R_{\rm par}L_{\mathfrak b})
\mathfrak g_1(L_{\rm std}+1/4),
\]
\[
E_{\rm rem}
:=\mathfrak g_1(L_{\rm std})
\left(
\frac12\norm{D^2f_0}_{C^0}E_{\rm disp}^2
+L_{\mathfrak b}E_{\rm disp}
\right).
\]
Set \(D_i^\alpha:=Z_i^\alpha-\gamma_i^0\).  Subtracting the integral equations gives
\[
\norm{D_i^\alpha}_{C([0,1])}\le E_{\rm disp}|\alpha|.
\]
Let \(z_i^\alpha\) solve
\begin{equation}
\dot z_i^\alpha
=Df_0(\gamma_i^0)z_i^\alpha
+\widetilde G_\alpha(\gamma_i^0),
\qquad
z_i^\alpha(0)=0,
\end{equation}
and set \(R_i^\alpha:=D_i^\alpha-z_i^\alpha\).  Taylor's formula and the Lipschitz bound on \(\widetilde G_\alpha\) give
\begin{equation}
\norm{R_i^\alpha}_{C([0,1])}
\le E_{\rm rem}|\alpha|^2.
\end{equation}
At \(t=1\), let \(\widetilde{\mathsf J}_{\rm end}\alpha\) denote the stack of the vectors \(z_i^\alpha(1)\).  The fundamental-matrix conjugacy yields
\begin{equation}
\mathsf J_{\rm end}
=\operatorname{diag}\bigl(DH(y_1^0),\ldots,DH(y_M^0)\bigr)
\widetilde{\mathsf J}_{\rm end}.
\end{equation}
Since \(\mathcal Q(\alpha):=\End(f+G_\alpha)\) is obtained by applying \(H\) to the stacked standard endpoints, Taylor's formula gives
\begin{equation}
\norm{\mathcal Q(\alpha)-Y-\mathsf J_{\rm end}\alpha}_{\R^m}
\le E_{\rm quad}|\alpha|^2,
\qquad |\alpha|\le r_0,
\end{equation}
where
\[
E_{\rm quad}
:=\sqrt M\left(
H_1E_{\rm rem}+\frac12H_2E_{\rm disp}^2
\right).
\]
Set
\begin{equation}
\label{eq:correction-radius}
r_*:=
\begin{cases}
r_0/2,&E_{\rm quad}=0,\\[1mm]
\displaystyle
\min\left\{r_0/2,\frac{\sigma_*}{8E_{\rm quad}}\right\},
&E_{\rm quad}>0.
\end{cases}
\end{equation}
Then \(r_*\le1/8\), and
\begin{equation}
\label{eq:linear-dominance}
\norm{\mathcal Q(\alpha)-Y-\mathsf J_{\rm end}\alpha}_{\R^m}
\le\frac{\sigma_*}{8}|\alpha|,
\qquad
|\alpha|\le r_*.
\end{equation}

\subsection{Active approximation radius and endpoint stability}

Define
\begin{equation}
\label{eq:final-ball}
R_{\rm safe}:=R_{\rm par}+2,
\qquad
R_*:=1+\max\{R_{\rm tr},R_{\rm safe}\}.
\end{equation}
Since \(H(\B_{R_{\rm tr}})=\B_{R_{\rm tr}}\) and \(H\) is the identity outside that ball, \(H(\overline\B_{R_{\rm safe}})\subset\B_{R_*}\).  Every trajectory of \(f_0+\widetilde G_\alpha\), \(|\alpha|\le r_*\), lies in \(\overline\B_{R_{\rm par}}\) and hence has distance at least two from \(\partial\B_{R_{\rm safe}}\).

Fix \(0<\eta\le1\), and suppose a locally Lipschitz field \(V_{\rm base}\) satisfies
\[
\norm{V_{\rm base}-f}_{C(\overline\B_{R_*})}\le\eta.
\]
For \(|\alpha|\le r_*\), define
\[
\widehat Z_i^\alpha(t)
:=H^{-1}(\Phi_{V_{\rm base}+G_\alpha}^t(x_i)).
\]
As long as \(\widehat Z_i^\alpha\) remains in \(\B_{R_{\rm safe}}\), its standard-coordinate equation differs from \eqref{eq:standard-parameter-ode} by at most \(H_{-1}\eta\).  Put
\begin{equation}
L_{\rm stab}:=L_{\rm std}+r_*L_{\mathfrak b},
\qquad
E_{\rm exit}:=H_{-1}\mathfrak g_1(L_{\rm stab}),
\qquad
E_{\rm end}:=\sqrt M\,H_1E_{\rm exit}.
\end{equation}
Because \(r_*\le r_0/2\), definition \eqref{eq:r0} gives \(r_*L_{\mathfrak b}\le1/8\), and hence \(L_{\rm stab}\le L_{\rm std}+1/8\).  Scalar Gronwall, stopped at the first exit from \(\B_{R_{\rm safe}}\), yields
\begin{equation}
\max_{1\le i\le M}
\norm{\widehat Z_i^\alpha-Z_i^\alpha}_{C([0,1])}
\le E_{\rm exit}\eta,
\qquad
\norm{\End(V_{\rm base}+G_\alpha)-\mathcal Q(\alpha)}_{\R^m}
\le E_{\rm end}\eta,
\end{equation}
provided \(E_{\rm exit}\eta\le1\).  The first inequality rules out a first exit because the reference path stays at distance two from the boundary; the second follows after applying \(H\) at the endpoints and stacking the \(M\) blocks.

Choose
\begin{equation}
\label{eq:eta-star}
\eta_*:=\frac12\min\left\{
1,\frac1{E_{\rm exit}},
\frac{\sigma_*r_*}{8E_{\rm end}}
\right\}.
\end{equation}
For every \(V_{\rm base}\) satisfying the bridge hypothesis and every \(|\alpha|\le r_*\), the stopped-trajectory argument gives
\begin{equation}
\label{eq:endpoint-stability-final}
\norm{\End(V_{\rm base}+G_\alpha)-\mathcal Q(\alpha)}_{\R^m}
\le\frac{\sigma_*r_*}{16}.
\end{equation}
The same estimates imply continuous dependence of the endpoint map on \(\alpha\) on \(\overline\B_{r_*}^m\).  In particular, all sample trajectories in both homotopies remain in \(\B_{R_*}\), which proves the active-radius assertion in Proposition~\ref{prop:quantitative-interface}.

\subsection{Degree closure and quantitative dependence}

On \(|\alpha|=r_*\), equations \eqref{eq:linear-dominance} and \eqref{eq:endpoint-stability-final} imply
\[
\norm{\End(V_{\rm base}+G_\alpha)-Y-\mathsf J_{\rm end}\alpha}_{\R^m}
\le\frac{3\sigma_*r_*}{16}
<|\mathsf J_{\rm end}\alpha|.
\]
Lemma~\ref{lem:degree-exactification} therefore gives \(\alpha_*\in\B_{r_*}^m\) satisfying \eqref{eq:bridge-exactness}.

For the bridge envelope, Appendix~\ref{app:finite-transitivity} bounds
\[
R_{\rm tr},\ \mathfrak H,\ \norm f_{C^{s_N}},\
\max_\ell\norm{h_\ell}_{C^{s_N}},\ M_U
\]
by \((C_NK/\rho)^{C_NK}\); hence \eqref{eq:sensitivity-tolerance} has the same bound for \(\varepsilon_*^{-1}\), while \eqref{eq:final-ball} gives the geometry-only bound for \(R_*\).  Moreover, \eqref{eq:pullback-g-envelope} gives
\[
A_{\mathfrak b}+L_{\mathfrak b}
\le
\left(C_N\frac K\rho\right)^{C_NK}\mathfrak b.
\]
The definitions of \(E_{\rm disp}\), \(E_{\rm rem}\), and \(E_{\rm quad}\) then show, after another enlargement of the same dimension-only constant, that
\[
E_{\rm disp}\le\Gamma_N(K,\rho)(1+\mathfrak b),
\qquad
E_{\rm rem}+E_{\rm quad}
\le\Gamma_N(K,\rho)(1+\mathfrak b)^2.
\]
Therefore \eqref{eq:r0} and \eqref{eq:correction-radius} give
\[
r_*^{-1}\le\Gamma_N(K,\rho)(1+\mathfrak b)^2.
\]
Since \(r_*L_{\mathfrak b}\le1/8\), the Gronwall factor is independent of \(\mathfrak b\), so \(E_{\rm exit}+E_{\rm end}\le\Gamma_N(K,\rho)\); thus \eqref{eq:eta-star} gives
\[
\eta_*^{-1}\le\Gamma_N(K,\rho)(1+\mathfrak b)^2.
\]
After one final enlargement of \(C_N\), these estimates are precisely \eqref{eq:bridge-geometry-outputs} and \eqref{eq:bridge-budget-envelope}.  This completes the proof of Proposition~\ref{prop:quantitative-interface}.

\section{Closure of quantitative bounds and coordinate pullback}
\label{app:quantitative-closure}

This appendix closes the bridge and ReLU approximation budgets under the envelope of Theorem~\ref{thm:quantitative-exactification}.

\paragraph{Closure of the quantitative envelope.}

\begin{lemma}[Closure of the coarse envelope]
\label{lem:coarse-envelope-closure}
Fix \(N\), let \(K\ge2\), \(0<\rho\le1\), and put \(X:=K/\rho\). Call a nonnegative quantity \(q\) coarsely admissible if \(q\le(c_NX)^{c_NK}\) for some constant \(c_N\ge2\) depending only on \(N\).  A finite collection of coarsely admissible quantities remains coarsely admissible, with one common dimension-only constant, after any fixed finite sequence of the following operations:
\begin{enumerate}[label=(\roman*)]
\item multiplication by a dimension-only constant, finite sums, finite products, and fixed positive powers;
\item multiplication by fixed powers of \(K\), \(M\), or \(m=NM\);
\item taking the reciprocal of a positive quantity defined as a dimension-only multiple of a finite minimum, provided the reciprocals of the entries of that minimum are already coarsely admissible;
\item applying \(q\mapsto\lceil q^{\kappa_N}\rceil\).
\end{enumerate}
The same conclusion holds for a sum or Euclidean aggregate indexed by \(1\le\ell\le m\), provided the summands are bounded by one common coarsely admissible quantity.
\end{lemma}

\begin{proof}
Since \(M\le K\), one has \(m=NM\le NK\), while \(X=K/\rho\ge K\ge2\). Consequently every fixed power of \(K\), \(M\), or \(m\) is absorbed by \((C_NX)^{C_NK}\) after increasing \(C_N\).  A fixed finite sum contributes only a fixed multiplicative factor, and a fixed product or positive power only multiplies the exponent by a fixed amount.  These changes are again absorbed by increasing \(C_N\).  Moreover,
\[
\left(\min_{1\le j\le J}q_j\right)^{-1}
=\max_{1\le j\le J}q_j^{-1}
\le\sum_{j=1}^Jq_j^{-1},
\qquad
\lceil q^{\kappa_N}\rceil\le1+q^{\kappa_N}.
\]
Finally, if \(0\le q_\ell\le q_{\max}\) for \(1\le\ell\le m\), then
\[
\sum_{\ell=1}^m q_\ell\le m q_{\max},
\qquad
\left(\sum_{\ell=1}^m q_\ell^2\right)^{1/2}
\le\sqrt m\,q_{\max}.
\]
Thus every listed operation preserves the claimed form.  Because only finitely many operations occur, the largest of the resulting dimension-only constants works simultaneously for all of them.
\end{proof}

Appendix~\ref{app:finite-transitivity} makes \(R_{\rm tr}\), \(\norm f_{C^{s_N}}\), and \(\max_\ell\norm{h_\ell}_{C^{s_N}}\) coarsely admissible.  By \eqref{eq:explicit-vector-budgets}--\eqref{eq:explicit-vector-width-strength},
\[
\mathcal S_h
\le\sqrt m\max_\ell\mathcal S_{R_{\rm tr},s_N}(h_\ell)
\]
is therefore coarsely admissible, and so is \(\mathfrak b_h=\max\{1,\sqrt2\mathcal S_h\}\).

The bridge geometry estimate \eqref{eq:bridge-geometry-outputs} makes \(R_*\) and \(\varepsilon_*^{-1}\) coarsely admissible.  Since the budget enters \eqref{eq:bridge-budget-envelope} only through the fixed square \((1+\mathfrak b_h)^2\), Lemma~\ref{lem:coarse-envelope-closure} also makes \(r_*^{-1}\) and \(\eta_*^{-1}\) coarsely admissible.  Applying the ReLU approximation formulas at \((R_{\rm tr},h_\ell,\varepsilon_*)\) and \((R_*,f,\eta_*)\) then gives the same conclusion for \(\mathcal P_h\), \(\mathcal P_f\), and \(\mathcal S_f\).  Since \(r_*\le1/8\), both final quantities \(\mathcal P_f+\mathcal P_h\) and \(\mathcal S_f+r_*\mathcal S_h\) are coarsely admissible.  Thus, after enlarging one dimension-only constant,
\begin{equation}
\label{eq:exact-ledger-coarse-envelope}
\max\left\{
R_{\rm tr},R_*,\norm f_{C^{s_N}},\max_\ell\norm{h_\ell}_{C^{s_N}},
\varepsilon_*^{-1},r_*^{-1},\eta_*^{-1},
\mathcal P_f+\mathcal P_h,\mathcal S_f+r_*\mathcal S_h
\right\}
\le\left(C_N\frac K\rho\right)^{C_NK}.
\end{equation}
Equations~\eqref{eq:exact-normalized-budgets} and~\eqref{eq:exact-ledger-coarse-envelope} prove \eqref{eq:normalized-joint}.

\subsection{Proof of the coordinate bound}

\begin{proof}[Proof of Corollary~\ref{cor:coordinate-bound}]
If all samples are stationary, the empty parameter list proves the claim.  Otherwise \(K\ge2\).  Put \(R_{\rm data}:=1+\max_{z\in\cV}\abs z\) and \(S_{\rm coord}(x):=x/(2R_{\rm data})\). The normalized vertices lie in \(\B_{1/2}\), and their separation is \(\sep(S_{\rm coord}\cV)=\sep(\cV)/(2R_{\rm data})\), so \(\sep(S_{\rm coord}\cV)^{-1}\le2\Gamma_{\mathrm{coord}}(\cV)\). Apply the normalized construction summarized in \eqref{eq:normalized-joint} to the data \(S_{\rm coord}\cD\), and pull its parameter list back by \(S_{\rm coord}\) and by time \(T\), exactly as in \eqref{eq:pullback-parameters}.  If \(\widehat\theta\) is the normalized list, this scalar pullback satisfies
\begin{equation}
\label{eq:raw-strength-pullback}
T\Str(\theta)
\le2R_{\rm data}\Str(\widehat\theta).
\end{equation}
Indeed, each pulled-back ridge contributes
\[
\abs{\widehat w_j}
\sqrt{\abs{\widehat a_j}^2+(2R_{\rm data})^2\widehat b_j^2}
\le
2R_{\rm data}\abs{\widehat w_j}
\sqrt{\abs{\widehat a_j}^2+\widehat b_j^2}
\]
after multiplication by \(T\).

Let \(C_{0,N}\) denote a dimension-only constant for which the normalized estimate \eqref{eq:normalized-joint} holds. Since \(K\le2M\) and \(\sep(S_{\rm coord}\cV)^{-1}\le2\Gamma_{\mathrm{coord}}(\cV)\), that estimate gives
\[
p(\theta),\ \Str(\widehat\theta)
\le
\bigl(4C_{0,N}M\Gamma_{\mathrm{coord}}(\cV)\bigr)^{2C_{0,N}M}.
\]
Choose the final constant \(C_N\) so that \(C_N\ge4C_{0,N}\) and \(C_N\ge2C_{0,N}+1\). Since \(M\Gamma_{\mathrm{coord}}(\cV)\ge1\), this implies
\[
\bigl(4C_{0,N}M\Gamma_{\mathrm{coord}}(\cV)\bigr)^{2C_{0,N}M}
\le
\bigl(C_NM\Gamma_{\mathrm{coord}}(\cV)\bigr)^{C_NM-1}.
\]
Finally,
\(2R_{\rm data}\le2\Gamma_{\mathrm{coord}}(\cV)
\le C_NM\Gamma_{\mathrm{coord}}(\cV)\).
Combining this with \eqref{eq:raw-strength-pullback} proves \eqref{eq:main-coordinate}.  Balancing now gives \(\norm{W_\theta}_{\mathrm F}^2+ \norm{A_\theta}_{\mathrm F}^2+ \norm{b_\theta}_2^2=2\Str(\theta)\), which proves \eqref{eq:coordinate-parameter-bound}.
\end{proof}

\bibliographystyle{siam}
\bibliography{ANODE_v20_references}
\end{document}